\documentclass[12pt]{amsart}
\pdfoutput=1

\usepackage{amsmath}
\usepackage{graphicx}
\usepackage{latexsym}
\usepackage{color}
\usepackage{amscd}
\usepackage[all]{xy}
\usepackage{enumerate}
\usepackage{hyperref}
\usepackage{subfigure}
\usepackage{soul}
\usepackage{comment}
\usepackage{epsfig}
\usepackage{epstopdf}
\usepackage{tikz}
\usepackage{marginnote}
\usepackage{lipsum}
\usepackage{stackengine}
\usepackage{multirow}

\newtheorem{thm}{Theorem}

\newtheorem{cor}[thm]{Corollary}

\newtheorem{theorem}[thm]{Theorem}
\newtheorem{lemma}[thm]{Lemma}

\newtheorem{question}[thm]{Question}

\theoremstyle{definition}

\newtheorem*{definition*}{Definition}

\newtheorem{remark}[thm]{Remark}

\newcommand{\N}{\mathbb{N}}
\newcommand{\Z}{\mathbb{Z}}

\newcommand{\M}{\operatorname{Mod}}

\def\cl{{\rm cl}}
\def\scl{{\rm scl}}

\def \eu{{\rm{e}}}

\newcommand{\QED}{\hfill \ensuremath{\Box}}

 \def\Z{{\mathbb{Z}}}
 
 \def\N{{\mathbb{N}}}

\begin{document}

\title[Geography of surface bundles over surfaces]
{Geography of surface bundles over surfaces}

\author[R. \.{I}. Baykur]{R. \.{I}nan\c{c} Baykur}
\address{Department of Mathematics and Statistics, University of Massachusetts, Amherst, MA 01003-9305, USA}
\email{inanc.baykur@umass.edu}

\author[M. Korkmaz]{Mustafa Korkmaz}
\address{Department of Mathematics, Middle East Technical University, 06800 Ankara, Turkey}
\email{korkmaz@metu.edu.tr}

\begin{abstract}
We construct symplectic surface bundles over surfaces with positive signatures for all but $18$ possible pairs of fiber and base genera.  Meanwhile,  we determine the commutator lengths of a few new mapping classes. 
\end{abstract}

\maketitle

\setcounter{secnumdepth}{2}
\setcounter{section}{0}

\section{Introduction} \label{sec:introduction}

Surface bundles over surfaces constitute an interesting family of  $4$--manifolds, which is amenable to techniques from different areas of mathematics, such as  algebraic geometry, symplectic topology and geometic group theory.  Let $\Sigma_g$ denote a closed orientable surface of genus $g$ and  let $\sigma$ denote the signature of a $4$--manifold.  Surface bundles with $\sigma=0$ are certainly easy to generate for any fiber and base genera;  in fact,  a $\Sigma_g$--bundle over $\Sigma_h$ has\, $\sigma = 0$ in many situations,  such as when $\pi_1(\Sigma_h)$ acts trivially on $H^*(\Sigma_g)$,  when the fibration is hyperelliptic (in particular when $g\leq 2$), or when the base genus $h \leq 1$\, \cite{ChernEtal, EndoHyperelliptic, Meyer}.  Further,  $\sigma \equiv 0$ (mod $4$) for any surface bundle  \cite{Meyer}. 

Our goal in this article is to provide a comprehensive answer to the following outstanding geography problem:

\smallskip
\textit{For which pairs of $(g,h) \in \N^2$ are there $\Sigma_g$--bundles over $\Sigma_h$ with  signature $\sigma > 0$?}

\smallskip
\noindent This problem on surface bundles has a long and rich history going back to the pioneering works of  Kodaira, Atiyah and Hirzebruch in the late 1960s \cite{Kodaira, Atiyah, Hirzebruch}, who produced the first examples of surface bundles with $\sigma > 0$ via branched coverings of products of complex curves, albeit for fairly large fiber or base genera. Endo's innovative work on signatures of surface bundles in the late 1990s \cite{Endo, EndoHyperelliptic}, via Meyer's cocycle and relations in the mapping class group, made it possible to approach this geography problem more systematically. Over the past two decades,  these methods have led to a myriad of examples of surface bundles with positive signatures; see e.g.  \cite{BaykurNonholomorphic, BeyazEtal, BryanDonagi, BryanDonagiStipsicz, CataneseRollenske, EKKOS, GonzalesEtal,  Lee, LeeEtal, Monden, Stipsicz, Zaal}. 

Our main result is an extensive  advancement in this line of research:

\begin{theorem} \label{main}
There exists a  symplectic $\Sigma_g$--bundle over $\Sigma_h$ with positive signature for every \,$g \geq 15$, $h=2$; \ $g \geq 9$, $h=3$; \ $g \geq 4$, $h =4$; \,and $g \geq 3 $, $h \geq 5$.
\end{theorem}

Since $\sigma=0$ when $g \leq 2$ or $h\leq 1$,  our theorem leaves out  $19$  possible $(g,h)$ pairs.  However,   Hamenst\"{a}dt showed  in \cite{Hamenstadt} that the Euler characteristic $\eu$ and the signature $\sigma$ of a surface bundle always satisfies the inequality $|\eu| \geq 3 \, |\sigma|$,  which allows us to rule out the existence of a surface bundle with positive signature for one (and just one) more possible pair: $(g,h)=(3,2)$.  While constructing surface bundles with positive signatures for all but $18$ possible pairs of $(g,h)$ is the best we could achieve at the time of writing,  it seems  plausible that variations of our techniques,  which we will discuss shortly,   may succeed in shrinking the gap even further.

All surface bundles we built in the theorem have signature $\sigma=4$.  Therefore,  for $\mathcal{M}_g$  the moduli space of genus--$g$ curves and $m_g$ denoting the minimal genus among the genera of all surfaces representing the infinite cyclic generator of $H_2(\mathcal{M}_g; \Z)/ \textrm{Tor}$,  with $g \geq 3$,  as observed in \cite{BryanDonagi}, we can  conclude from our results that  $m_g=2$ for any $g \geq 15$, and  we have the estimates  $m_g=2$ or $3$ when $9 \leq g \leq 14$,  $m_g=2, 3$ or $4$ when $4 \leq g \leq 8$, and $m_g=3, 4$ or $5$ when $g=3$.

We  describe all but one of our surface bundles in Theorem~\ref{main} via explicit monodromy factorizations in the mapping class group $\M(\Sigma_g)$.\,(The remaining example uses a semi-stable holomorphic fibration due to  Catanese--Corvaja--Zannier in \cite{CataneseEtal} as an ingredient.) The breakthrough in our understanding of relations that generate these small surface bundles with positive signatures is due to shorter commutator expressions we are able to obtain for both products of commutators themselves and multi-twists in the mapping class group. In particular, Theorem~\ref{thm:tsuboish}, the proof of which adapts an ingenious argument of Tsuboi in  \cite{Tsuboi} and Burago, Ivanov and Polterovich in \cite{BuragoEtal},  allows one to derive examples with base genus $h=2$ and $3$ from those over higher genera surfaces.  Leveraging these ingredients,  we  calculate  in Corollaries~\ref{cor:cl} and~\ref{cor:scl} the commutator and stable commutator lengths of a few new mapping classes, in particular providing new answers to the Kirby Problem 2.13(b)~\cite{Kirby}.

One of the motivations for our work is to better understand how the geography of surface bundles compare to that of symplectic $4$--manifolds and compact complex surfaces. While all surface bundles with positive signatures admit  symplectic forms a l\'{a} Thurston,  their total spaces do not necessarily admit complex structures. \linebreak In fact,  by the first author's work in \cite{BaykurNonholomorphic}, the surface bundles we construct in this article yield infinitely many such examples for all possible fiber and base genera except for less than two dozen pairs; see Remark~\ref{Nonholomorphic}.

A particularly interesting comparative geography problem is for the border case of Bogomolov--Miyaoka--Yau inequality \cite{Bogomolov, Miyaoka, Yau}. By Yau's celebrated solution of the Calabi conjecture \cite{Yau}, any compact complex surface of general type with $\eu = 3 \, \sigma$ is a complex ball quotient. These constitute a rather small but very interesting class of complex surfaces, which can not contain any surface bundles  \cite{Liu,  Kapovich}. In contrast, it is still not known whether there are symplectic $4$--manifolds of general type with $\eu = 3 \, \sigma$ that are not complex ball quotients, leading to the more specific question:

\begin{question}
Is there a symplectic surface bundle over a surface with $\eu = 3\,\sigma$?
\end{question}

\noindent This amounts to asking in particular whether there is a $\Sigma_g$--bundle over $\Sigma_h$ with positive signature for $(g,h)=(4,2)$ ---where our example with $(g,h)=(4,4)$ gets provokingly close! And more generally,  it is part of the bigger question on the existence of any further constraints on the geography of surface bundles with positive signatures,  while obviously the very examples in this article limit much wilder constraints to be expected.

\bigskip
\noindent 
\textit{Basic conventions:\,} 
All  manifolds and maps we consider in this article are smooth. We denote a compact orientable surface of genus $g$ with $b$ boundary components by $\Sigma_{g}^b$,  whereas we omit $b$ when there is no boundary. 
We denote by $\mathrm{Diff}^+(\Sigma_g^b)$ the group of orientation--preserving diffeomorphisms 
 $\Sigma_g^b\to \Sigma_g^b$ that restrict to the identity in a collar neighborhood of the boundary. The \textit{mapping class group} of $\Sigma_g^b$ is defined as $\M(\Sigma_{g}^b):=\pi_0(\mathrm{Diff}^+(\Sigma_g^b))$. Our products of mapping classes act on curves starting with the rightmost factor. 
  Whenever we study relations in the mapping class group of $\Sigma_g^b$, 
  we  consider the curves on $\Sigma_g^b$ and the elements in $\mathrm{Diff}^+(\Sigma_g^b)$ 
  only up to isotopy and we denote their isotopy classes by the same symbols. 
  We denote  by $t_c$ the right-handed, or the \textit{positive Dehn twist}, along a simple closed curve 
  $c$ on a surface $\Sigma_g^b$. For any $A$ and $B$ in $\M(\Sigma_g^b)$, we let $[A,B]:=A B A^{-1} B^{-1}$ denote their commutator, and $A^{B}:=BA B^{-1}$ denote the conjugate of $A$  by $B$. We denote by $\lfloor r \rfloor$ the largest integer less than or equal to the real number $r$.

\noindent 
\textit{Further conventions:\,} 
By a \textit{genus--$g$ surface bundle $(X,f)$ over a genus--$h$ surface} we mean a smooth locally trivial $\Sigma_g$--bundle $f\colon X \to \Sigma_h$, where $X$ is an oriented $4$--manifold. A \textit{monodromy factorization} for $(X,f)$ with $b$ disjoint sections $\{S_j\}$ of self-intersections $S_j \cdot S_j=-k_j$ is a relation of the form 
\[  [A_1, B_1] \cdots [A_h, B_h] =  t_{\delta_1}^{k_1} \cdots t_{\delta_b}^{k_b}  
   \ \ \ \ \text{ in }\M(\Sigma_g^b),
\]
where $A_i, B_i$ are general elements in $\M(\Sigma_g^b)$ and $\{\delta_j\}$ are boundary parallel curves along distinct boundary components of $\Sigma_g^b$. Finally, for any relator $W=1$ in $\M(\Sigma_g^b)$ we define the signature $\sigma(W)$ as the algebraic sum of the signatures of the relators that are used to derive it from the trivial word with respect to Dehn twist generators \cite{Endo, EndoNagami}.  We refer the reader to \cite{FarbMargalit, EKKOS, BaykurKorkmazMonden, BaykurMargalit, 
EndoNagami, Monden} for the general background on surface bundles, monodromy factorizations, mapping class group relations, and signatures.

\clearpage
\section{Shorter expressions for products of commutators} \label{sec:comm4comm}

There are many situations when a given product of commutators in a group can be re-expressed as a product of less number of commutators.  For example,  the famous Hall-Witt identity for arbitrary $a, b, c$ in a group $G$ can be  arranged to read
\[ [[a,b], c^b] \, [[b,c], a^c] = [b^a, [c,a]]. \] 
Here are a couple of other instances that might be less well--known:

\begin{lemma} \label{lem:commutators}
For $a, b, c,d$ any elements in a group $G$, the following hold:
\begin{enumerate}[\rm (i)]
\item \, $[a,b]\, [b,c]\,[c,d]\,[d,a]=[a^{-1} c,  b^{-1}d]^{ab}$, and
\item \, $\prod_{i=1}^k[a_i, b_i]= [\prod_{i=1}^k  a_i ,  \prod_{i=1}^k  b_i ]$,  \, if
 $[a_i, a_j]=[a_i, b_j]=[b_i,b_j]=1$ for all $i\neq j$. 
\end{enumerate}
\end{lemma}

\begin{proof} Both identities can be checked in a straightforward fashion:
\begin{eqnarray*}
    [a,b]\, [b,c]\,[c,d]\,[d,a] &=& aba^{-1}\underline{b^{-1} \, b}cb^{-1}\underline{c^{-1} \, c}dc^{-1}\underline{d^{-1} \, d}ad^{-1}a^{-1} \\
	&=&  aba^{-1} \, cb^{-1} \, dc^{-1} \, ad^{-1}a^{-1}  	 \\
	&=&  aba^{-1} cb^{-1} dc^{-1} ad^{-1} \, \underline{b b^{-1}} \, a^{-1}  \\
	&=& (ab) \, a^{-1} c  \, b^{-1} d  \, (a^{-1} c)^{-1}\, (b^{-1} \, d)^{-1} \, (a b)^{-1} \\
	&=& [a^{-1} c , b^{-1}  d ]^{a b}, 
\end{eqnarray*}
where we have underlined the canceling pairs.  Likewise,
\begin{eqnarray*}
   \prod_{i=1}^k[a_i, b_i]
    &=& a_1b_1a_1^{-1}b_1^{-1} \, a_2b_2a_2^{-1}b_2^{-1} \cdots a_kb_ka_k^{-1}b_k^{-1} \\
     &=& a_1a_2\cdots a_k \, b_1a_1^{-1}b_1^{-1} \, b_2a_2^{-1}b_2^{-1} \cdots b_ka_k^{-1}b_k^{-1} \\
     &=& a_1a_2\cdots a_k \, b_1b_2 \cdots b_k \, a_1^{-1}b_1^{-1} \, a_2^{-1}b_2^{-1} \cdots a_k^{-1}b_k^{-1} \\
     &=& a_1a_2\cdots a_k \, b_1b_2 \cdots b_k \, a_1^{-1}a_2^{-1} \cdots a_k^{-1} \, b_1^{-1}b_2^{-1} \cdots b_k^{-1} \\
     &=& a_1a_2\cdots a_k \, b_1b_2 \cdots b_k \, a_k^{-1}\cdots a_2^{-1} a_1^{-1} \, b_k^{-1}\cdots b_2^{-1} b_1^{-1} \\
     &=& [\prod_{i=1}^k  a_i ,  \prod_{i=1}^k  b_i ],
\end{eqnarray*}
where in each one of the intermediate steps we have repeatedly used only  the given commutativity relations. 
\end{proof}

Recall that a conjugate of a commutator is again a commutator,  so  that
all the commutator identities we have listed so far, in fact, describe a product of commutators as a single commutator.  The first identity in the lemma appears in the literature as early as in \cite{Neumann} and contains two special cases which appear more frequently: 
\[   [a,b]\, [b,c]\,[c,a] =[a^{-1} c,  b^{-1}a]^{ab} \]
which one derives by taking $d=a$ in Lemma~\ref{lem:commutators}(i), whereas  taking $d=1$ we get
\begin{equation} \label{eqn:usefulcomm}
[a,b]\, [b,c]= [a^{-1} c,  b^{-1}]^{ab}.
\end{equation}

The second identity in the lemma is perhaps more contemporary but was clearly already known to experts  \cite{Tsuboi, BuragoEtal} (more on this below).

For monodromy factorizations of surface bundles,  commutator identities as above allow one to derive new surface bundles over surfaces of smaller genera. Our next theorem,  the proof of which is leveraging a beautiful argument of  Tsuboi  in \cite{Tsuboi} and Burago, Ivanov and Polterovich in \cite{BuragoEtal},  shows that one can moreover lower the base genus dramatically at the expense of increasing  the fiber genus:

\begin{theorem} \label{thm:tsuboish}
 Let $(X,f)$ be a genus--$g$ surface bundle over a genus--$h$ surface with  a section of 
 self-intersection zero. Then there also exist surface bundles
 \begin{enumerate}[\rm (i)]
	\item\, $(X',f')$ of fiber genus $g'=gh$ and base genus $2$,  for $h \geq 2$,  and
	\item\,  $(X'',f'')$ of fiber genus $g''=g \lfloor \frac{h+1}{2} \rfloor$ and base genus $3$,  for $h \geq 3$,
 \end{enumerate}
 also with sections of self-intersection zero and signatures 
  $\sigma(X') = \sigma(X'') = \sigma(X)$. 
  Further,  given an explicit monodromy factorization for $(X,f)$ with a self-intersection zero section $S$,  
  we can  explicitly describe the monodromy factorization of $(X',f')$ and $(X'',f'')$ with self-intersection 
  zero sections $S'$ and $S''$,  respectively. 
\end{theorem}

\begin{proof}
Let $a_i, b_i$,  for $i=1, \ldots, h$,  be elements of $\textrm{Diff}^{+}(\Sigma_g^1)$ which restrict to the identity in some collar neighborhood of $\partial \Sigma_g^1$.  Let $\phi=\prod_{i=1}^h[a_i, b_i]$. 
That is to say, we have the following relation: 
\begin{equation} \label{eqn:Tinput}
 \phi=[a_1, b_1] \cdots [a_h, b_h] \ \text{ in } \M(\Sigma_g^1) \,  ,
\end{equation}
where we simply denote the corresponding mapping classes by the same letters.

\medskip
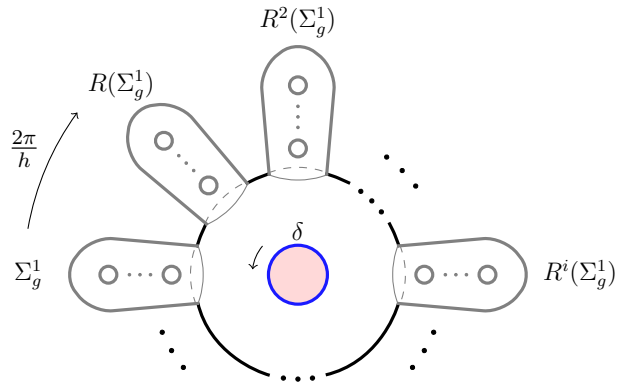
\begin{figure}[ht]
\begin{tikzpicture}[scale=0.8]
\begin{scope} [xshift=0cm, yshift=0cm, scale=0.7]
  \draw[very thick, rounded corners=1pt] (2.41,0.66) -- (2.28,1)--(2.16,1.24);
  \draw[very thick, rounded corners=1pt, rotate=45] (2.41,0.66) -- (2.28,1)--(2.16,1.24);
  \draw[very thick, rounded corners=1pt, rotate=90] (2.41,0.66) -- (2.28,1)--(2.16,1.24);
  \draw[very thick, rounded corners=1pt, rotate=135] (2.41,0.66) -- (2.28,1)--(2.16,1.24);
  \draw[very thick, rounded corners=5pt, rotate=-90] (2.38,0.66)..controls (2.1,1.84) and (1,2.3)..(0.67, 2.39);
  \draw[very thick, rounded corners=5pt, rotate=-180] (2.38,0.66)..controls (2.1,1.84) and (1,2.3)..(0.67, 2.39);
   \filldraw[color=blue!90, fill=red!15, very thick] (0,0) circle (0.7cm);
    \filldraw[rotate=0] (0,-2.48) circle (0.05cm);
    \filldraw[rotate=8] (0,-2.48) circle (0.05cm);
    \filldraw[rotate=-8] (0,-2.48) circle (0.05cm);
   \filldraw[rotate=135] (0,-2.48) circle (0.05cm);
   \filldraw[rotate=143] (0,-2.48) circle (0.05cm);
   \filldraw[rotate=127] (0,-2.48) circle (0.05cm);
  \draw[very thick, gray, rounded corners=8pt, yshift=0, xshift=0cm, rotate=0] (2.35,0.68) -- (4.9,0.9)--(5.6,0)--(4.9,-0.9) --(2.35,-0.68);
  \draw[gray, very thick] (3,0) circle (0.2cm);
  \draw[gray, very thick] (4.5,0) circle (0.2cm);
  \filldraw[gray] (3.5,0) circle (0.03cm);
  \filldraw[gray] (3.75,0) circle (0.03cm);  
  \filldraw[gray] (4,0) circle (0.03cm);
  \draw[gray, rounded corners=4pt, rotate=0] (2.38,0.72) -- (2.25,0.3)--(2.25,-0.3) --(2.38,-0.72);
  \draw[gray, dashed, rounded corners=4pt, rotate=0] (2.42,0.71) -- (2.55,0.3)--(2.55,-0.3) --(2.42,-0.71);  
  \draw[very thick, gray, rounded corners=8pt, yshift=0, xshift=0cm, rotate=90] (2.35,0.68) -- (4.9,0.9)--(5.6,0)--(4.9,-0.9) --(2.35,-0.68);
  \draw[very thick, gray, rotate=90] (3,0) circle (0.2cm); 
  \draw[very thick, gray, rotate=90] (4.5,0) circle (0.2cm);
  \filldraw[gray, rotate=90] (3.5,0) circle (0.03cm);
  \filldraw[gray, rotate=90] (3.75,0) circle (0.03cm);  
  \filldraw[gray, rotate=90] (4,0) circle (0.03cm);
  \draw[gray, rounded corners=4pt, rotate=90] (2.38,0.72) -- (2.25,0.3)--(2.25,-0.3) --(2.38,-0.72);
  \draw[gray, dashed, rounded corners=4pt, rotate=90] (2.42,0.71) -- (2.55,0.3)--(2.55,-0.3) --(2.42,-0.71);   
 \draw[very thick,  gray, rounded corners=8pt, yshift=0, xshift=0cm, rotate=135] (2.35,0.68) -- (4.9,0.9)--(5.6,0)--(4.9,-0.9) --(2.35,-0.68);
  \draw[very thick, gray, rotate=135] (3,0) circle (0.2cm);
  \draw[very thick, gray,  rotate=135] (4.5,0) circle (0.2cm);
  \filldraw[gray, rotate=135] (3.5,0) circle (0.03cm);
  \filldraw[gray, rotate=135] (3.75,0) circle (0.03cm);  
  \filldraw[gray, rotate=135] (4,0) circle (0.03cm);
    \draw[gray, rounded corners=4pt, rotate=135] (2.38,0.72) -- (2.25,0.3)--(2.25,-0.3) --(2.38,-0.72);
  \draw[gray, dashed, rounded corners=4pt, rotate=135] (2.42,0.71) -- (2.55,0.3)--(2.55,-0.3) --(2.42,-0.71);
  \draw[very thick, gray, rounded corners=8pt, yshift=0, xshift=0cm, rotate=180] (2.35,0.68) -- (4.9,0.9)--(5.6,0)--(4.9,-0.9) --(2.35,-0.68);
  \draw[very thick, gray, rotate=180] (3,0) circle (0.2cm);
  \draw[very thick, gray, rotate=180] (4.5,0) circle (0.2cm);
  \filldraw[gray, rotate=180] (3.5,0) circle (0.03cm);
  \filldraw[gray, rotate=180] (3.75,0) circle (0.03cm);  
  \filldraw[gray,rotate=180] (4,0) circle (0.03cm);
    \draw[gray, rounded corners=4pt, rotate=180] (2.38,0.72) -- (2.25,0.3)--(2.25,-0.3) --(2.38,-0.72);
  \draw[gray, dashed, rounded corners=4pt, rotate=180] (2.42,0.71) -- (2.55,0.3)--(2.55,-0.3) --(2.42,-0.71);
  \filldraw[rotate=127] (0,-3.5) circle (0.05cm);
  \filldraw[rotate=135] (0,-3.5) circle (0.05cm);
  \filldraw[rotate=143] (0,-3.5) circle (0.05cm);
  \filldraw[rotate=67] (0,-3.5) circle (0.05cm);
  \filldraw[rotate=59] (0,-3.5) circle (0.05cm);
  \filldraw[rotate=51] (0,-3.5) circle (0.05cm);
  \filldraw[rotate=-67] (0,-3.5) circle (0.05cm);
  \filldraw[rotate=-59] (0,-3.5) circle (0.05cm);
  \filldraw[rotate=-51] (0,-3.5) circle (0.05cm);
\draw[->, rounded corners=5pt, rotate=67] (-1.5, 6.327)..controls (-0.5, 6.557) and (0.5, 6.557).. (1.5,6.327);
\node[rotate=0] at (-6.5,3) {$\frac{2\pi}{h}$};
\draw[->, rounded corners=2pt, rotate=67] (0.3, 1.04)..controls (0.1, 1.1) and (-0.1, 1.1).. (-0.3,1.04);
 \node[scale=0.8] at (-6.4,0) {$\Sigma_g^1$};
 \node[scale=0.8] at (-4.2,4.5) {$R(\Sigma_g^1)$};
 \node[scale=0.8] at (0,6) {$R^2(\Sigma_g^1)$};
 \node[scale=0.8] at (6.7,0) {$R^i(\Sigma_g^1)$};
  \node[scale=0.8] at (0,1.04) {$\delta$};
\end{scope}
\end{tikzpicture}
\caption{The rotation $R$.} \label{fig:rotation}
\end{figure}
\smallskip

Now let $R$ be the clockwise $\frac{2\pi}{h}$--rotation of $\Sigma_{g'}^1$, with $g'=gh$ as  illustrated in Figure~\ref{fig:rotation},  followed by a counter-clockwise $\frac{2\pi}{h}$--rotation of $\partial \Sigma_{g'}^1$ supported in a small collar neighborhood of its boundary.  We take an embedding $\Sigma_g^1  \hookrightarrow   \Sigma_{g'}^1$ with image  as shown in Figure~\ref{fig:rotation},  away from the support of the above boundary rotation.  With this identification of $\Sigma_g^1$ with a subsurface of $\Sigma_{g'}^1$, we can then define $\tilde{\phi},  \tilde{a}_i, \tilde{b}_i \in \textrm{Diff}^{+}(\Sigma_{g'}^1)$
by extending each $\phi,  a_i, b_i$ as the identity on $\Sigma_{g'}^1 \setminus \Sigma_g^1$. 
We thus have the relation
\begin{equation}\label{eqn:Tinitial}
\tilde{\phi} = [\tilde{a}_1, \tilde{b}_1] \cdots [\tilde{a}_h, \tilde{b}_h]  \ \text{ in } \M(\Sigma_{g'}^1) \,  . 
\end{equation}

\noindent 
\underline{\textit{A two-commutator identity:}} Reviewing the details of Tsuboi's algebraic argument in  \cite{Tsuboi} (cf. \cite{BuragoEtal}) will be essential for  the remaining part of the proof of our theorem.  Let $C_i:= [\tilde{a}_i, \tilde{b}_i]$, for all $i=1, \ldots, h$,  so $\tilde{\phi} = \prod_{i=1}^h C_i$.   Following \cite{Tsuboi}, we set
\[ P:=  \prod_{i=1}^{h} (C_1^{R^{h-i}} \cdots C_i^{R^{h-i}}) 
\, . \]

There are quite a few commutativity relations which are important to note here.  First of all,  $[\tilde{a}_i^{R^p},  \tilde{b}_j^{R^q}]=[\tilde{a}_i^{R^p},  \tilde{a}_j^{R^q}]=[\tilde{b}_i^{R^p},  \tilde{b}_j^{R^q}]=1$ for any $p \neq q$ since each pair of diffeomorphisms in these commutators have disjoint supports in $\Sigma_{g'}^1$. 
It follows that $[C_i^{R^p}, C_j^{R^q}]=1$ for any $i, j$ and $p \neq q$.  In turn, 
$[C_1^{R^{h-i}} \cdots C_i^{R^{h-i}} \, , \, C_1^{R^{h-j}} \cdots C_j^{R^{h-j}}]=1$ whenever $i \neq j$,  so one can spell out the parenthetical factors in the above product expression of $P$ in any ---and in particular,  in reversed--- order.  Last but not least, 
since $R^h$ is identity on the compact support of any $\tilde{a}_i^{R^p}$ and $\tilde{b}_i^{R^p}$ (even though it is isotopic to a boundary parallel Dehn twist on $\Sigma_{g'}^1$),  we have $[R^h,  \tilde{a}_i^{R^p}]= [R^h,  \tilde{b}_i^{R^p}]=1$,  and thus  $[R^h,  C_i^{R^p}]=1$ for any $i$ and $p$, and in turn, $ (C_i^{R^p})^{R^h}=C_i^{R^p}$.

We have the product expressions
\[ P^{-1}= \prod_{i=1}^{h}(C_1^{R^{h-i}} \cdots C_i^{R^{h-i}})^{-1} 
= \tilde{\phi}^{-1} \ \prod_{i=1}^{h-1}(C_1^{R^{h-i}}  \cdots C_i^{R^{h-i}})^{-1}  \]
and 
\[
P^R= \prod_{i=1}^{h}(C_1^{R^{h-i+1}} \cdots C_{i}^{R^{h-i+1}})
= \prod_{i=0}^{h-1}(C_1^{R^{h-i}} \cdots C_{i+1}^{R^{h-i}}) 
= C_1  \prod_{i=1}^{h-1}(C_1^{R^{h-i}} \cdots C_{i+1}^{R^{h-i}}) . 
 \]
Therefore,  
\begin{eqnarray*}
  [P^{-1}, R]  
     &=&  P^{-1} P^R                  \\
     &=& \tilde{\phi}^{-1} \  \prod_{i=1}^{h-1}(C_1^{R^{h-i}}  \cdots C_i^{R^{h-i}})^{-1} \,  \  \cdot C_1  \prod_{i=1}^{h-1} (C_1^{R^{h-i}} \cdots C_{i+1}^{R^{h-i}})   \\
     &=& \tilde{\phi}^{-1} \ C_1 \  \prod_{i=1}^{h-1}
     ((C_1^{R^{h-i}}  \cdots C_i^{R^{h-i}})^{-1} \, (C_1^{R^{h-i}} \cdots C_{i+1}^{R^{h-i}}))   \\    
     &=& \tilde{\phi}^{-1} \ C_1 \  \prod_{i=1}^{h-1} C_{i+1}^{R^{h-i}}  \\
     &=& \tilde{\phi}^{-1} \  \prod_{i=1}^{h} C_i^{R^{h-i+1}} \\
&=&   \tilde{\phi}^{-1} \ [\, \prod_{i=1}^h \tilde{a}_i^{R^{h-i+1}}   \, ,  \, 
\prod_{i=1}^h \tilde{b}_i^{R^{h-i+1}}  ]
\end{eqnarray*}
where we repeatedly used the commutativity relations mentioned above and  
invoked Lemma~\ref{lem:commutators}(ii) at the final step.

Setting $A:=\prod_{i=1}^h \tilde{a}_i^{R^{h-i+1}}$ and $B:=\prod_{i=1}^h \tilde{b}_i^{R^{h-i+1}}$ we arrive at the two-commutator identity
\begin{equation}\label{eqn:Toutput1}
\tilde{\phi}=[A, B][R,  P^{-1}] \text{ in } \M(\Sigma_{g'}^1) \,  .
\end{equation}

\medskip
\noindent \underline{\textit{A variation:}}
Let $\tilde{\phi} = \prod_{i=1}^h C_i$ and $R$ be as above.  This time set
\[ 
 Q:=  \prod_{i=0}^{h-1} (C_1^{R^{i}} \cdots C_{h-i}^{R^{i}}) = \tilde{\phi} \ 
  \prod_{i=1}^{h-1} (C_1^{R^{i}} \cdots C_{h-i}^{R^{i}})
 \,  . \]
Note that by the commutativity relations mentioned above,  the parenthetical terms in this product can also be spelled out in any order.
So by using the equality $(C_1^{-1})^{R^h}=C_1^{-1}$, we have 
\[
(Q^{-1})^R= \prod_{i=0}^{h-1}(C_1^{R^{i+1}} \cdots C_{h-i}^{R^{i+1}})^{-1}
= \prod_{i=1}^{h}(C_1^{R^{i}} \cdots C_{h-i+1}^{R^{i}})^{-1} 
=  C_1^{-1} \ \prod_{i=1}^{h-1}(C_1^{R^{i}} \cdots C_{h-i+1}^{R^{i}})^{-1} \, .
 \]
 It follows that  for  $[Q, R]  = Q R Q^{-1} R^{-1} = Q (Q^{-1})^R $,  applying the same arguments as earlier, we have   
\begin{eqnarray*}
  [Q, R]   
     &=&  \tilde{\phi} \  \prod_{i=1}^{h-1}(C_1^{R^{i}}  \cdots C_{h-i}^{R^{i}}) \,  \ \cdot  C_1^{-1}  \prod_{i=1}^{h-1} (C_1^{R^{i}} \cdots C_{h-i+1}^{R^{i}})^{-1}     
      \\
     &=&  \tilde{\phi} \   C_1^{-1} \ \prod_{i=1}^{h-1}(C_1^{R^{i}}  \cdots C_{h-i}^{R^{i}}) \,  (C_1^{R^{i}} \cdots C_{h-i+1}^{R^{i}})^{-1}  \\
     &=&  \tilde{\phi} \  C_1^{-1} \ \prod_{i=1}^{h-1}(C_1^{R^{i}}  \cdots C_{h-i}^{R^{i}}) \,  
     (C_{h-i+1}^{R^{i}})^{-1}
     (C_1^{R^{i}} \cdots C_{h-i}^{R^{i}})^{-1}   \\     
     &=&  \tilde{\phi} \  C_1^{-1} \  \prod_{i=1}^{h-1} (C_{h-i+1}^{C_1^{R^{i}} \cdots \, C_{h-i}^{R^{i}} R^{i}})^{-1}  \\  
          &=&  \tilde{\phi} \   \prod_{i=1}^{h} (C_{h-i+1}^{C_1^{R^{i}} \cdots \, C_{h-i}^{R^{i}} R^{i}})^{-1}  \\  
          &=&  \tilde{\phi} \   \prod_{i=1}^{h} (C_{i}^{C_1^{R^{h-i+1}} \cdots \, C_{i-1}^{R^{h-i+1}} R^{h-i+1}})^{-1}   \\  
&=&   \tilde{\phi} \  [\, \prod_{i=1}^h (\tilde{a}_i^{C_1^{R^{h-i+1}} \cdots \, C_{i-1}^{R^{h-i+1}} R^{h-i+1}})^{-1}  \, ,  \, 
\prod_{i=1}^h (\tilde{b}_i^{C_1^{R^{h-i+1}} \cdots \, C_{i-1}^{R^{h-i+1}} R^{h-i+1}})^{-1}  ] \, .
\end{eqnarray*}
(Here we take $C_1^{R^{h-i+1}} \cdots \, C_{i-1}^{R^{h-i+1}}=1$ if $i=1$, and so on.)
Setting 
\[ A':=\prod_{i=1}^h (\tilde{a}_i^{C_1^{R^{h-i+1}} \cdots \, C_{i-1}^{R^{h-i+1}} R^{h-i+1}})^{-1}  \text{ and }
B':=\prod_{i=1}^h (\tilde{b}_i^{C_1^{R^{h-i+1}} \cdots \, C_{i-1}^{R^{h-i+1}} R^{h-i+1}})^{-1} \]
gives us the equality
\begin{equation}\label{eqn:Toutput2}
\tilde{\phi}=[Q,R][B', A'] \text{ in } \M(\Sigma_{g'}^1) \,  .
\end{equation}

\medskip
\noindent \underline{\textit{A three-commutator identity:}} Assume that for $j=1,2$ we have 
\[ \phi_j:= \prod_{i=1}^{h_j} [a_i(j), b_i(j)] \ \text{ in } \M(\Sigma_g^1)  \]
and let $h$ be the maximum of $h_1$ and $h_2$. Possibly after adding trivial commutators,  we can express both as  $\phi_j:= \prod_{i=1}^{h} [a_i(j), b_i(j)]$.  Running Tsuboi's trick for $j=1$ and its above variation for $j=2$, respectively, we obtain  two identities
\begin{eqnarray*}
\tilde{\phi}_1 &=& [A, B] [R, P^{-1}] =[R, P^{-1}] [A, B]^{[P^{-1},R]}  \\ 
\tilde{\phi}_2 &=& [Q,R][B', A']= [B', A']^{[Q,R]}[Q,R]  
\end{eqnarray*}
in $\M(\Sigma_{g'}^1)$.
Here, the triples $A, B, P$ and $A', B', Q$ are  determined by diffeomorphisms coming from the entries of the commutators in $\phi_1$ and $\phi_2$, respectively,  but $R$ is the same diffeomorphism of $\Sigma_{g'}^1$.  By 
the special case of Lemma~\ref{lem:commutators}(i),  we  have 
\begin{eqnarray*}
\tilde{\phi}_2 \tilde{\phi}_1 &=& [B', A']^{[Q,R]}[Q,R][R, P^{-1}] [A, B]^{[P^{-1},R]} \\
&=& [B',A']^{[Q,R]}
[Q^{-1}P^{-1}, R^{-1}]^{QR} 
[A, B]^{[P^{-1},R]}\, . 
\end{eqnarray*}

Relabeling  the conjugated commutator entries,  we get a new three commutator expression
\begin{equation} \label{eqn:Toutput3}
\tilde{\phi}_2 \tilde{\phi}_1 = [A_1, B_1][A_2,B_2][A_3,B_3]  \text{ in } \M(\Sigma_{g'}^1) \, .
\end{equation}

We note that here $g'=gh$ for $h=\textrm{max}\{h_1, h_2\}$ and not $h_1+h_2$.

\medskip
\noindent \underline{\textit{Constructions of $(X',f')$ and $(X'',f'')$:}}
There is a monodromy factorization for $(X,f)$ with a section $S$ of self-intersection zero of the following form:
\begin{equation} \label{eqn:monodX}
 [a_1, b_1] \cdots [a_h, b_h] =  1  \text{ in } \M(\Sigma_g^1) \,  . 
\end{equation}
Taking $\phi=1$ in~\eqref{eqn:Tinput},  and in turn getting $\tilde{\phi}=1$ in \eqref{eqn:Tinitial},  the equality~\eqref{eqn:Toutput1} becomes:
\begin{equation} \label{eqn:monodX'}
[A, B][R, P^{-1}]=1 \text{ in } \M(\Sigma_{g'}^1) \,  .
\end{equation} 
Prescribed by this relation is  the surface bundle $(X',f')$ with fiber genus $g'=gh$ and base genus $2$,  along with a section $S'$ of self-intersection zero. 

On the other hand, if we take $\phi_2=[a_1, b_1] \cdots [a_{k}, b_{k}]$ and $\phi_1= [a_{k+1}, b_{k+1}] \cdots [a_h, b_h]$ for $k :=  \lfloor \frac{h+1}{2} \rfloor$ in our construction of the three-commutator identity,  the equality~\eqref{eqn:Toutput3} becomes
\begin{equation} \label{eqn:monodX''}
[A_1, B_1][A_2,B_2][A_3,B_3] =1 \text{ in } \M(\Sigma_{g''}^1) \,  ,
\end{equation} 
where $g''= g  k$.  Prescribed by this relation is  the surface bundle $(X'',f'')$ with fiber genus $g''= g \lfloor \frac{h+1}{2} \rfloor$ and base genus $3$,  along with a section $S''$ of self-intersection zero. 

The monodromy factorization of a surface bundle is derived from the trivial word in the mapping class group of a fiber using some sequence of basic relators between Dehn twists.  By \cite{EndoNagami},  the signature of this surface bundle can be expressed as an algebraic sum of the signature contributions of these basic relators,  and the result is independent of the sequence.  Moreover the signature contribution of a relator and its conjugate are the same, so it suffices to look at a few basic relators. 

Post factum, after embedding a positive factorization in $\Sigma_g^1$ into the mapping class group of another surface as we did via the embedding $\Sigma_g \hookrightarrow \Sigma_{g'}$,  we can in fact regard it to be obtained from the trivial word in $\M(\Sigma_{g'}^1)$ using the same basic relators in $\M(\Sigma_{g'}^1)$.  (See next section for more on this.) Therefore, the initial positive factorization in our monodromy constructions,  namely \eqref{eqn:monodX}  has the same signature as the positive factorization \eqref{eqn:Tinput} with $\phi=1$.  Pivotal to our construction is that after that point,  in our derivation of the positive factorizations~\eqref{eqn:monodX'} and~\eqref{eqn:monodX''} all the relations we have used can be easily seen to be only commutativity and conjugation relators,  along with insertion/removal of canceling pairs,  all of which have zero signature contribution. 
Hence,  we have  $\sigma(X)=\sigma(X')= \sigma(X'')$,  as promised.
\end{proof}


\medskip

\section{Shorter expressions for products of Dehn twists} \label{sec:comm4DTs}

We now turn to expressing products of positive Dehn twists as products of small numbers of commutators, 
while using only the relators (between Dehn twist generators) in the mapping class group with non-negative signature contributions.\footnote{This seemingly unnatural restriction is due to our aspirations to build surface bundles with positive signatures via monodromy factorizations we will obtain in the final section by combining the relators we get here with other known relators that have negative signatures.)} Despite this resriction, several of our commutator expressions will make it possible for us to calculate the precise (and positive) commutator and stable commutator lengths of a few new mapping classes.  These are discussed at the very end of this section.

More explicitly, our focus here is on expressing a product $P$ of positive Dehn twists as a product $C$ of a few commutators, so that the signature of the relator $P^{-1} C=1$ is positive. Following \cite{Endo, EndoNagami}, we will take the infinite presentation of $\M(\Sigma_g^b)$ with Dehn twists along all curves as the generators, and the relators between them.  We review some of these relators and their signatures first.

\subsection{Basic relators and signatures in the mapping class group} \

We have the obvious relator $t_a t_a^{-1}=1$ for any $a$. If two curves $a$ and $b$ are disjoint, then we have the \textit{commutativity} relator $t_at_bt_a^{-1}t_b^{-1}=1$. Similarly, for simple closed curves 
$a$ and $b$ intersecting transversely at one point, there is the \emph{braid} relator $t_at_bt_at_b^{-1}t_a^{-1}t_b^{-1}=1$. All these basic relators have signature $\sigma=0$. It then follows that for any $A, B \in \M(\Sigma_g^b)$, 
the \emph{conjugation} relator $A^{-1}B^{-1}A^B B=1$ has $\sigma=0$ as well. This has several implications. For one, if one induces a relator $W'=1$ in $\M(\Sigma_{g'}^{b'})$ from a given relator $W=1$ in $\M(\Sigma_g^b)$ through some embedding $\Sigma_g^b \hookrightarrow \Sigma_{g'}^{b'}$, then $\sigma(W')=\sigma(W)$, and $\sigma(W')$ is in fact independent of the embedding.

There are two basic relators with positive signatures that are of importance to us. First, for a null-homotopic curve $a$ in $\Sigma_g^b$, the well-known relator $t_{a}^{-1}=1$ has $\sigma=+1$. One implication of this is the following: Say we have a monodromy factorization for a surface bundle with sections expressed by a relator 
$W:=C  \, t_{\delta_1}^{-k_1}   t_{\delta_2}^{-k_2} \cdots t_{\delta_b}^{-k_b} =1 $ 
in $\M(\Sigma_g^b)$, where $C$ is some product of commutators. Suppose that a relator $W'=1$ in $\M(\Sigma_g^{b-1})$ is derived from $W$ by capping off the boundary component
$\delta_i$ with a disk. Then $\sigma(W')=\sigma(W)+k_i$ by the above reasoning.

The point guard in our game is the \emph{lantern relator} in $\M(\Sigma_0^4)$
\[ t_{\delta_1}^{-1}t_{\delta_2}^{-1}t_{\delta_3}^{-1}t_{\delta_4}^{-1} t_xt_yt_z =1  ,
\]
where the curves $x, y, z, \delta_i$ are as shown in Figure~\ref{fig:lanternve4-holedtorus}. This relator also has $\sigma=+1$, which now follows from the facts we laid out  above, once you embed $\Sigma_0^4 \hookrightarrow \Sigma_g^b$. 

\medskip
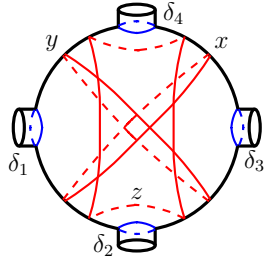
\begin{figure}[h]
\begin{tikzpicture}[scale=0.7]
\begin{scope} [xshift=0cm, yshift=0cm, scale=0.7]
    \draw[very thick ] (-2.75,0.467) arc (170:99.7:2.8);
  \draw[very thick ] (2.75,0.47) arc (10:80.3:2.8);
  \draw[very thick, rotate=90 ] (-2.75,0.467) arc (170:99.7:2.8);
 \draw[very thick, rotate=-90 ] (2.75,0.47) arc (10:80.3:2.8);
  \draw[very thick, rounded corners=10pt] (-0.5, 2.73)--(-0.5 ,3.2);
   \draw[very thick, rounded corners=10pt] (0.5, 2.73)--(0.5 ,3.2);
   \draw[very thick, rounded corners=10pt, rotate=90] (-0.5, 2.73)--(-0.5 ,3.2);
   \draw[very thick, rounded corners=10pt, rotate=90] (0.5, 2.73)--(0.5 ,3.2);
   \draw[very thick, rounded corners=10pt, rotate=180] (-0.5, 2.73)--(-0.5 ,3.2);
   \draw[very thick, rounded corners=10pt, rotate=180] (0.5, 2.73)--(0.5 ,3.2);
   \draw[very thick, rounded corners=10pt, rotate=270] (-0.5, 2.73)--(-0.5 ,3.2);
   \draw[very thick, rounded corners=10pt, rotate=270] (0.5, 2.73)--(0.5 ,3.2);
 \draw[very thick, rotate=0] (0,3.2) ellipse (0.5 and 0.15);
 \draw[very thick, rotate=90] (0,3.2) ellipse (0.5 and 0.15);
 \draw[very thick, rotate=180] (0,3.2) ellipse (0.5 and 0.15);
 \draw[very thick, rotate=270] (0,3.2) ellipse (0.5 and 0.15);
  \draw[thick, blue, rounded corners=3pt, yshift=-8, xshift=0cm, rotate=0] (-0.5,3) -- (-0.2,2.85)--(0.2,2.85) --(0.5,3);
  \draw[thick,  blue, dashed, rounded corners=3pt, yshift=-8, xshift=0cm, rotate=0] (-0.5,3.05) -- (-0.2,3.17)--(0.2,3.17) --(0.5,3.05);
  \draw[thick, blue, rounded corners=3pt, xshift=8, rotate=90] (-0.5,3) -- (-0.2,2.85)--(0.2,2.85) --(0.5,3);
  \draw[thick,  blue, dashed, rounded corners=3pt, xshift=8, rotate=90] (-0.5,3.05) -- (-0.2,3.17)--(0.2,3.17) --(0.5,3.05);
  \draw[thick, blue, rounded corners=3pt, yshift=8, xshift=0cm, rotate=180] (-0.5,3) -- (-0.2,2.85)--(0.2,2.85) --(0.5,3);
  \draw[thick,  blue, dashed, rounded corners=3pt, yshift=8, xshift=0cm, rotate=180] (-0.5,3.05) -- (-0.2,3.17)--(0.2,3.17) --(0.5,3.05);
 \draw[thick, blue, rounded corners=3pt, xshift=-8, rotate=270] (-0.5,3) -- (-0.2,2.85)--(0.2,2.85) --(0.5,3);
  \draw[thick,  blue, dashed, rounded corners=3pt, xshift=-8, rotate=270] (-0.5,3.05) -- (-0.2,3.17)--(0.2,3.17) --(0.5,3.05);   
   \draw[thick, red, rounded corners=12pt, rotate=45] (-2.8,-0.03) -- (-1.8,-0.25)--(1.8,-0.25) --(2.8,-0.03);
   \draw[thick, red, dashed, rounded corners=16pt, rotate=45] (-2.8,0.03) -- (-1,0.25)--(1,0.25) --(2.8,0.03);
   \draw[thick, red, rounded corners=12pt, rotate=135] (-2.8,-0.03) -- (-1.8,-0.25)--(1.8,-0.25) --(2.8,-0.03);
    \draw[thick, red, dashed, rounded corners=16pt, rotate=135] (-2.8,0.03) -- (-1,0.25)--(1,0.25) --(2.8,0.03);
    \draw[thick, red, dashed, rounded corners=8pt, rotate=0] (-1.3,2.47) -- (-0.5,2.1)--(0.5,2.1) --(1.3,2.47);
  \draw[thick, red, dashed, rounded corners=8pt, rotate=180] (-1.3,2.47) -- (-0.5,2.1)--(0.5,2.1) --(1.3,2.47);
     \draw[thick, red, rounded corners=12pt, rotate=0] (-1.3,2.47) -- (-1,1.5)--(-1,-1.5) --(-1.3,-2.47);
    \draw[thick, red, rounded corners=12pt, rotate=180] (-1.3,2.47) -- (-1,1.5)--(-1,-1.5) --(-1.3,-2.47);
 \node[scale=0.8] at (0,-1.8) {$z$};
 \node[scale=0.8] at (-2.3,2.3) {$y$};
 \node[scale=0.8] at (2.3,2.3) {$x$};
 \node[scale=0.8] at (-3.2,-1) {$\delta_1$};
 \node[scale=0.8] at (-0.9,-3.2) {$\delta_2$};
 \node[scale=0.8] at (3.2,-0.9) {$\delta_3$};
 \node[scale=0.8] at (1,3.1) {$\delta_4$};
\end{scope}
\end{tikzpicture}
\caption{The lantern curves on $\Sigma_0^4$.}  \label{fig:lanternve4-holedtorus}
\end{figure}
\smallskip

\smallskip
\subsection{Multitwists as products of a few commutators} \

In the next several lemmas, we are going to express various \emph{multitwists}, i.e. products of positive 
Dehn twists about disjoint curves, as small number of commutators using only the relators 
(with non-negative signatures) we have listed earlier. Our relations are supported in $\Sigma_g^b$, for $g=2$, $g=3$ and $g=5$ (and varying $b$), respectively, where, importantly, we always have $b>0$, so the same relations can be embedded into any $\M(\Sigma_{g'}^{b'})$ for $g' \geq g$.

We are going to make repeated use of the following simple but highly useful observation,  which has been well-exploited in several prior works; e.g.  \cite{ BaykurKorkmazMonden, EKKOS, Hamada, Harer,  Korkmaz, KorkmazOzbagci, Monden}.
Let $a_1, a_2,\ldots,  a_m$ and $b_1, b_2,\ldots, b_m$ be curves on $\Sigma_g^b$ 
such that there is a diffeomorphism $F \in \textrm{Diff}^+(\Sigma_g^b)$ with $F(a_i)=b_i$ 
for all $i=1, \ldots, m$.  We then have the one-commutator expression
\begin{equation} \label{1comm}
t_{a_1}^{k_1} \cdots t_{a_m}^{k_m} \, t_{b_m}^{-k_m} \cdots t_{b_1}^{-k_1} = [ t_{a_1}^{k_1} \cdots t_{a_m}^{k_n}, F ]
\ \ \ \ \text{ in } \M(\Sigma_g^b)\, .
\end{equation}
One can easily check that the corresponding relator has signature zero.

In the proof of the next lemma,  our arguments work for any power of the same (multi)twist that is of particular interest to us,   so we will derive the commutator expressions for all of them at once.  

\smallskip
\begin{lemma}  \label{lem:genus=2}
 For $\delta_1,\delta_2$ any two boundary components of $\Sigma_2^3$, and $n$ any positive integer,  
 there is a relator 
 \[  
 t_{\delta_1}^{-n} t_{\delta_2}^{-n} \prod_{i=1}^N C_i(n)=1  \ \ \ \text{ in }  \M(\Sigma_2^3)
 \]
 with  $N=\lfloor (n+3)/2 \rfloor$ commutators $\{C_i(n)\}$ and signature $2n$.
 \end{lemma}
\begin{proof}
Consider $\Sigma_2^3$ in Figure~\ref{fig:g=2decomm}.
By the lantern relation, we have
  \begin{eqnarray*}
	t_{a_1}t_{a_2}t_{a_3}t_{\delta_1} &=& t_{x_3}t_{x_1}t_{x_2},\\
	t_bt_{a_2}t_{a_3}t_{\delta_2} &=& t_{y_1}t_{y_2}t_{y_3}
  \end{eqnarray*}
 so that 
  \begin{eqnarray*}
	t_{a_1}t_{a_2}t_{a_3}t_{\delta_1}t_{x_2}^{-1} &=& t_{x_3}t_{x_1},\\
	t_bt_{a_2}t_{a_3}t_{\delta_2}t_{y_3}^{-1} &=& t_{y_1}t_{y_2}.
  \end{eqnarray*}
  Hence,
 \begin{eqnarray*}
	t_{a_1}^nt_{a_2}^nt_{a_3}^nt_{\delta_1}^nt_{x_2}^{-n} &=& (t_{x_3}t_{x_1})^n
		= \left( \prod _{i=1}^{n}  t_{x_1}^{i-1} t_{x_3} {t_{x_1}^{-i+1}} \right) t_{x_1}^n 
		=\left( \prod _{i=1}^{n}  \left( t_{x_3}\right) ^{ {t_{x_1}^{i-1}}} \right) t_{x_1}^n, \\
	t_b^nt_{a_2}^nt_{a_3}^nt_{\delta_2}^nt_{y_3}^{-n} &=& (t_{y_1}t_{y_2})^n
		= \left(  \prod _{i=1}^{n}   t_{y_2}^{i-1} t_{y_1}  t_{y_2}^{-i+1} \right) t_{y_2}^n 
		= \left(  \prod _{i=1}^{n}   \left( t_{y_1} \right)^{  t_{y_2}^{i-1}} \right) t_{y_2}^n .
  \end{eqnarray*}
 By multiplying the last two equalities, we get
 \begin{eqnarray*}
	t_b^nt _{a_1}^nt_{a_2}^{2n}t_{a_3}^{2n}t_{\delta_1}^n  t_{\delta_2}^n
	&=& \left(  \prod _{i=1}^{n} (t_{x_3}t_{y_1})^{V^{i-1}}\right)  t_{x_1}^n t_{y_2}^n     t_{x_2}^n t_{y_3}^n,
  \end{eqnarray*}
  where $V= t_{x_1}t_{y_2}$. Thus
 \begin{eqnarray*}
	t_{\delta_1}^n  t_{\delta_2}^n
   	&=& \left(  \prod _{i=1}^{n} \left( t_{x_3}t_{y_1} t_b^{-1}t_{a_3}^{-1}\right)^{V^{i-1}}\right)   
	(t_{x_1}^n  t_{y_2}^n t_{a_1}^{-n} t_{a_2}^{-n}     t_{x_2}^n t_{y_3}^nt_{a_2}^{-n}t_{a_3}^{-n}).
  \end{eqnarray*}
 Since there is a diffeomorphism of $\Sigma_2^3$ mapping $(x_1,y_2, a_1,a_2)$  to
 $(a_2,a_3,x_2,y_3)$, the product
 \( t_{x_1}^n  t_{y_2}^n t_{a_1}^{-n} t_{a_2}^{-n}     t_{x_2}^n t_{y_3}^nt_{a_2}^{-n}t_{a_3}^{-n}
 \) 
 is a commutator. 
 
 \medskip
\begin{figure}[ht]
\begin{tikzpicture}[scale=0.8]

\begin{scope} [xshift=0cm, yshift=0cm, scale=0.6]
  \draw[very thick,rounded corners=14pt] (3,2) --(-5.23,2) ;
  \draw[very thick,rounded corners=14pt] (-5.23,-2) -- (3,-2) ;
 \draw[very thick, dashed, rounded corners=10pt] (-5.2,2)..controls (-4.7, 1.5) and (-4.7,-1.5) .. (-5.2,-2) ;
 \draw[very thick,rounded corners=10pt] (-5.2,2)..controls (-5.7, 1.5) and (-5.7,-1.5) .. (-5.2,-2) ;
 \draw[very thick, xshift=-2.6cm] (0,0) circle [radius=0.6cm];
  \draw[very thick, xshift=0cm] (0,0) circle [radius=0.6cm];
  \draw[very thick, rounded corners=2pt] (3, 0.6)..controls (2.7, 0.6) and (2.2,0.6) .. (2,0)..controls (2.2, -0.6) and  (2.7,-0.6)..(3, -0.6) ;
  \draw[very thick, xshift=0cm] (3,1.3) ellipse (0.2 and 0.7);
  \draw[very thick, xshift=0cm] (3,-1.3) ellipse (0.2 and 0.7);
  \draw[thick, blue, rounded corners=4pt, xshift=-1.3cm] (-0.7,-0.03) -- (-0.4,-0.13)--(0.4,-0.13) --(0.7,-0.03);
 \draw[thick,  blue, dashed, rounded corners=4pt, xshift=-1.3cm] (-0.7,0.03) -- (-0.4,0.13)--(0.4,0.13) --(0.7,0.03);
  \draw[thick, blue, rounded corners=4pt, xshift=1.3cm] (-0.7,-0.03) -- (-0.4,-0.13)--(0.4,-0.13) --(0.7,-0.03);
 \draw[thick,  blue, dashed, rounded corners=4pt, xshift=1.3cm] (-0.7,0.03) -- (-0.4,0.13)--(0.4,0.13) --(0.7,0.03);
\draw[thick,  blue, rounded corners=4pt, xshift=-2.6cm ] (-0.03,0.6) -- (-0.13, 1)--(-0.13,1.6) --(-0.03,2);
\draw[thick,  blue, dashed, rounded corners=4pt, xshift=-2.6cm ] (0.03,0.6) -- (0.13, 1)--(0.13,1.6) --(0.03,2);
\draw[thick,  blue, rounded corners=4pt, xshift=-2.6cm ] (-0.03,-0.6) -- (-0.13, -1)--(-0.13,-1.6) --(-0.03,-2);
\draw[thick,  blue, dashed, rounded corners=4pt, xshift=-2.6cm ] (0.03,-0.6) -- (0.13, -1)--(0.13,-1.6) --(0.03,-2);
\node[scale=0.9] at (-3.2,1.3) {$a_1$};
\node[scale=0.9] at (-1.2,-0.6) {$a_2$};
\node[scale=0.9] at (1.4,-0.6) {$a_3$};
\node[scale=0.9] at (3.7,1.3) {$\delta_1$};
\node[scale=0.9] at (-3.2,-1.3) {$b$};
\node[scale=0.9] at (3.7,-1.3) {$\delta_2$};
\node[scale=0.9] at (-6,1.6) {$\delta_3$};
\end{scope}

\begin{scope} [xshift=6.5cm, yshift=0cm, scale=0.6]
   \draw[very thick,rounded corners=14pt] (3,2) --(-5.23,2) ;
  \draw[very thick,rounded corners=14pt] (-5.23,-2) -- (3,-2) ;
 \draw[very thick, dashed, rounded corners=10pt] (-5.2,2)..controls (-4.7, 1.5) and (-4.7,-1.5) .. (-5.2,-2) ;
 \draw[very thick,rounded corners=10pt] (-5.2,2)..controls (-5.7, 1.5) and (-5.7,-1.5) .. (-5.2,-2) ;
 \draw[very thick, xshift=-2.6cm] (0,0) circle [radius=0.6cm];
  \draw[very thick, xshift=0cm] (0,0) circle [radius=0.6cm];
  \draw[very thick, rounded corners=2pt] (3, 0.6)..controls (2.7, 0.6) and (2.2,0.6) .. (2,0)..controls (2.2, -0.6) and  (2.7,-0.6)..(3, -0.6) ;
  \draw[very thick, xshift=0cm] (3,1.3) ellipse (0.2 and 0.7);
  \draw[very thick, xshift=0cm] (3,-1.3) ellipse (0.2 and 0.7);
\draw[thick, red, rounded corners=8pt, xshift=1.3cm] (-3.45,0.4) -- (-2.8,1)--(-1.3,1.4)--(0.2,1) --(0.85,0.4);
 \draw[thick, red, dashed, rounded corners=8pt, xshift=1.3cm] (-3.4,0.35) -- (-2.8,0.8)--(-1.3,1.1)--(0.2,0.8) --(0.8,0.32);
\draw[thick, red, rounded corners=4pt, xshift=-1.3cm] (3.45,-2) ..controls (3.2,-1.1) and (-0.7,-1.6).. (-0.8,-0.35);
\draw[thick, red, rounded corners=4pt, xshift=-1.3cm] (3.5,-0.4)--(3.1,-0.75) ..controls (2.2,-1.2) and (0,-1.2)..(0.8,-0.35);
\draw[thick, red, dashed, rounded corners=4pt, xshift=1.3cm ] (0.97,-0.5) -- (1.1, -0.9)--(1.1,-1.6) --(0.97,-2);
\draw[thick, red, dashed, rounded corners=4pt, xshift=-1.3cm] (-0.78,-0.27) -- (-0.4,-0.2)--(0.4,-0.2) --(0.77, -0.27); 
  \draw[thick, blue, rounded corners=4pt, xshift=-1.3cm] (-0.7,-0.03) -- (-0.4,-0.13)--(0.4,-0.13) --(0.7,-0.03);
 \draw[thick,  blue, dashed, rounded corners=4pt, xshift=-1.3cm] (-0.7,0.03) -- (-0.4,0.13)--(0.4,0.13) --(0.7,0.03);
\draw[thick,  blue, rounded corners=4pt, xshift=-2.6cm ] (-0.03,0.6) -- (-0.13, 1)--(-0.13,1.6) --(-0.03,2);
\draw[thick,  blue, dashed, rounded corners=4pt, xshift=-2.6cm ] (0.03,0.6) -- (0.13, 1)--(0.13,1.6) --(0.03,2);

\node[scale=0.9] at (1.7,1.3) {$x_1$};
\node[scale=0.9] at (-1.4,-1.5) {$y_2$};
\end{scope}

\begin{scope} [xshift=0cm, yshift=-3cm, scale=0.6]
  \draw[very thick,rounded corners=14pt] (3,2) --(-5.23,2) ;
  \draw[very thick,rounded corners=14pt] (-5.23,-2) -- (3,-2) ;
 \draw[very thick, dashed, rounded corners=10pt] (-5.2,2)..controls (-4.7, 1.5) and (-4.7,-1.5) .. (-5.2,-2) ;
 \draw[very thick,rounded corners=10pt] (-5.2,2)..controls (-5.7, 1.5) and (-5.7,-1.5) .. (-5.2,-2) ;
 \draw[very thick, xshift=-2.6cm] (0,0) circle [radius=0.6cm];
  \draw[very thick, xshift=0cm] (0,0) circle [radius=0.6cm];
  \draw[very thick, rounded corners=2pt] (3, 0.6)..controls (2.7, 0.6) and (2.2,0.6) .. (2,0)..controls (2.2, -0.6) and  (2.7,-0.6)..(3, -0.6) ;
  \draw[very thick, xshift=0cm] (3,1.3) ellipse (0.2 and 0.7);
  \draw[very thick, xshift=0cm] (3,-1.3) ellipse (0.2 and 0.7);
\draw[thick, red, rounded corners=4pt, xshift=1.3cm] (-3.45,2) ..controls (-3.2,1.1) and (0.7,1.6).. (0.8,0.35);
\draw[thick, red, rounded corners=4pt, xshift=1.3cm] (-3.5,0.4)--(-3.1,0.75) ..controls (-2.2,1.2) and (0,1.2)..(-0.8,0.35);
\draw[thick, red, dashed, rounded corners=4pt, xshift=1.7cm ] (-3.93,0.5) -- (-4.03, 1)--(-4.03,1.6) --(-3.93,2);
\draw[thick, red, dashed, rounded corners=4pt, xshift=1.3cm] (-0.78,0.27) -- (-0.4,0.2)--(0.4,0.2) --(0.77, 0.27);
\draw[thick,  red, rounded corners=4pt, xshift=3.9cm ] (-3.93,-0.6) -- (-4.03, -1)--(-4.03,-1.6) --(-3.93,-2);
\draw[thick,  red, dashed, rounded corners=4pt, xshift=3.9cm] (-3.87,-0.6) -- (-3.77, -1)--(-3.77,-1.6) --(-3.87,-2);
  \draw[thick, blue, rounded corners=4pt, xshift=-1.3cm] (-0.7,-0.03) -- (-0.4,-0.13)--(0.4,-0.13) --(0.7,-0.03);
 \draw[thick,  blue, dashed, rounded corners=4pt, xshift=-1.3cm] (-0.7,0.03) -- (-0.4,0.13)--(0.4,0.13) --(0.7,0.03);
  \draw[thick, blue, rounded corners=4pt, xshift=1.3cm] (-0.7,-0.03) -- (-0.4,-0.13)--(0.4,-0.13) --(0.7,-0.03);
 \draw[thick,  blue, dashed, rounded corners=4pt, xshift=1.3cm] (-0.7,0.03) -- (-0.4,0.13)--(0.4,0.13) --(0.7,0.03);
\node[scale=0.9] at (1.2,1.5) {$x_2$};
\node[scale=0.9] at (0.8,-1.3) {$y_3$};
\end{scope}

\begin{scope} [xshift=6.5cm, yshift=-3cm, scale=0.6]
  \draw[very thick,rounded corners=14pt] (3,2) --(-5.23,2) ;
  \draw[very thick,rounded corners=14pt] (-5.23,-2) -- (3,-2) ;
 \draw[very thick, dashed, rounded corners=10pt] (-5.2,2)..controls (-4.7, 1.5) and (-4.7,-1.5) .. (-5.2,-2) ;
 \draw[very thick,rounded corners=10pt] (-5.2,2)..controls (-5.7, 1.5) and (-5.7,-1.5) .. (-5.2,-2) ;
 \draw[very thick, xshift=-2.6cm] (0,0) circle [radius=0.6cm];
  \draw[very thick, xshift=0cm] (0,0) circle [radius=0.6cm];
  \draw[very thick, rounded corners=2pt] (3, 0.6)..controls (2.7, 0.6) and (2.2,0.6) .. (2,0)..controls (2.2, -0.6) and  (2.7,-0.6)..(3, -0.6) ;
  \draw[very thick, xshift=0cm] (3,1.3) ellipse (0.2 and 0.7);
  \draw[very thick, xshift=0cm] (3,-1.3) ellipse (0.2 and 0.7);
\draw[thick, red, rounded corners=4pt, xshift=3.9cm ] (-3.93,0.6) -- (-4.03, 1)--(-4.03,1.6) --(-3.93,2);
 \draw[thick, red, dashed, rounded corners=4pt, xshift=3.9cm] (-3.87,0.6) -- (-3.77, 1)--(-3.77,1.6) --(-3.87,2);
\draw[thick, red, rounded corners=8pt, xshift=1.3cm] (-3.45,-0.4) -- (-2.8,-0.8)--(-1.3,-1.2)--(0.2,-0.8) --(0.85,-0.4);
 \draw[thick, red, dashed, rounded corners=8pt, xshift=1.3cm] (-3.4,-0.35) -- (-2.8,-0.7)--(-1.3,-1)--(0.2,-0.6) --(0.8,-0.35);
  \draw[thick, blue, rounded corners=4pt, xshift=1.3cm] (-0.7,-0.03) -- (-0.4,-0.13)--(0.4,-0.13) --(0.7,-0.03);
 \draw[thick,  blue, dashed, rounded corners=4pt, xshift=1.3cm] (-0.7,0.03) -- (-0.4,0.13)--(0.4,0.13) --(0.7,0.03);
\draw[thick,  blue, rounded corners=4pt, xshift=-2.6cm ] (-0.03,-0.6) -- (-0.13, -1)--(-0.13,-1.6) --(-0.03,-2);
\draw[thick,  blue, dashed, rounded corners=4pt, xshift=-2.6cm ] (0.03,-0.6) -- (0.13, -1)--(0.13,-1.6) --(0.03,-2);

\node[scale=0.9] at (-0.7,1.3) {$x_3$};
\node[scale=0.9] at (1.6,-1.2) {$y_1$};
\end{scope}
\end{tikzpicture}
\caption{The curves on $\Sigma_2^3$ used in the proof of Lemma~\ref{lem:genus=2}.} 
\label{fig:g=2decomm}
\end{figure}
\smallskip

Note also that since there is a diffeomorphism of $\Sigma_2^3$ mapping $(x_3,y_1, b,a_3)$  to
 $(a_3,b,y_1,x_3)$, the product
\[
  t_{x_3}t_{y_1} t_b^{-1}t_{a_3}^{-1} (t_{x_3}t_{y_1} t_b^{-1}t_{a_3}^{-1})^V
 \]
 is a commutator as well.  It follows that 
if $n=2k$ is even, then
  \begin{eqnarray*}
\prod _{i=1}^{n} \left( t_{x_3}t_{y_1}t_b^{-1}t_{a_3}^{-1}\right)^{V^{i-1}}
    &=& \prod _{i=1}^{k} \left( t_{x_3}t_{y_1}t_b^{-1}t_{a_3}^{-1}  
    	(t_{x_3}t_{y_1} t_b^{-1}t_{a_3}^{-1})^V    \right)^{V^{2i-2}}	
  \end{eqnarray*}
is a product of $k$ commutators, and if $n=2k+1$, then 
   \begin{eqnarray*}
 		\prod _{i=1}^{n} 		\left( t_{x_3}t_{y_1}t_b^{-1}t_{a_3}^{-1}\right)^{V^{i-1}}
    	&=& \left(  \prod _{i=1}^{n-1}   \left( t_{x_3}t_{y_1}t_b^{-1}
				t_{a_3}^{-1}\right)^{V^{i-1}}\right)  
    	(t_{x_3}t_{y_1} t_b^{-1}t_{a_3}^{-1})^{V^{2k}}
  \end{eqnarray*}
is a product of $k+1$ commutators. 

Therefore, $t_{\delta_1}^n  t_{\delta_2}^n=\prod_{i=1}^N C_i(n)$ for $N=\lfloor (n+3)/2 \rfloor$,  
  where $C_i(n)$ are commutators, varying with $n$. As we have only employed signature zero relators and $2n$ lantern relators, the relator 
 $t_{\delta_1}^{-n}  t_{\delta_2}^{-n} \prod_{i=1}^N C_i(n)=1$ has \mbox{$\sigma= 2n$.}
 \end{proof}

\medskip
\begin{figure}[ht]
\begin{tikzpicture}[scale=0.8]

\begin{scope} [xshift=0cm, yshift=0cm, scale=0.6]
  \draw[very thick,rounded corners=14pt] (5.23,-2) --(-4,-2)--(-4.83,0)-- (-4,2) --(5.23,2) ;
 \draw[very thick, rounded corners=10pt] (5.2,2)..controls (4.7, 1.5) and (4.7,-1.5) .. (5.2,-2) ;
 \draw[very thick,rounded corners=10pt] (5.2,2)..controls (5.7, 1.5) and (5.7,-1.5) .. (5.2,-2) ;
 \draw[very thick, xshift=-2.6cm] (0,0) circle [radius=0.6cm];
 \draw[very thick, xshift=0cm] (0,0) circle [radius=0.6cm];
 \draw[very thick, xshift=2.6cm] (0,0) circle [radius=0.6cm];
  \draw[thick, blue, rounded corners=4pt, xshift=-1.3cm] (-0.7,-0.03) -- (-0.4,-0.13)--(0.4,-0.13) --(0.7,-0.03);
 \draw[thick,  blue, dashed, rounded corners=4pt, xshift=-1.3cm] (-0.7,0.03) -- (-0.4,0.13)--(0.4,0.13) --(0.7,0.03);
  \draw[thick, blue, rounded corners=4pt, xshift=1.3cm] (-0.7,-0.03) -- (-0.4,-0.13)--(0.4,-0.13) --(0.7,-0.03);
 \draw[thick,  blue, dashed, rounded corners=4pt, xshift=1.3cm] (-0.7,0.03) -- (-0.4,0.13)--(0.4,0.13) --(0.7,0.03);
\draw[thick,  blue, rounded corners=4pt, xshift=2.6cm ] (-0.03,0.6) -- (-0.13, 1)--(-0.13,1.6) --(-0.03,2);
\draw[thick,  blue, dashed, rounded corners=4pt, xshift=2.6cm ] (0.03,0.6) -- (0.13, 1)--(0.13,1.6) --(0.03,2);
 \draw[thick, blue, rounded corners=4pt, xshift=-3.9cm] (-0.7,-0.03) -- (-0.4,-0.13)--(0.4,-0.13) --(0.7,-0.03);
 \draw[thick,  blue, dashed, rounded corners=4pt, xshift=-3.9cm] (-0.7,0.03) -- (-0.4,0.13)--(0.4,0.13) --(0.7,0.03);
\draw[thick,  blue, rounded corners=4pt, xshift=2.6cm ] (-0.03,-0.6) -- (-0.13, -1)--(-0.13,-1.6) --(-0.03,-2);
\draw[thick,  blue, dashed, rounded corners=4pt, xshift=2.6cm ] (0.03,-0.6) -- (0.13, -1)--(0.13,-1.6) --(0.03,-2);
\node[scale=0.9] at (-3.8,-0.6) {$a_1$};
\node[scale=0.9] at (-1.2,-0.6) {$a_2$};
\node[scale=0.9] at (1.4,-0.6) {$a_3$};
\node[scale=0.9] at (3.3,1.3) {$a_4$};
\node[scale=0.9] at (3.2,-1.3) {$b$};
\end{scope}

\begin{scope} [xshift=8.5cm, yshift=0cm, scale=0.6]
  \draw[very thick,rounded corners=14pt] (5.23,-2) --(-4,-2)--(-4.83,0)-- (-4,2) --(5.23,2) ;
 \draw[very thick, rounded corners=10pt] (5.2,2)..controls (4.7, 1.5) and (4.7,-1.5) .. (5.2,-2) ;
 \draw[very thick,rounded corners=10pt] (5.2,2)..controls (5.7, 1.5) and (5.7,-1.5) .. (5.2,-2) ;
 \draw[very thick, xshift=-2.6cm] (0,0) circle [radius=0.6cm];
 \draw[very thick, xshift=0cm] (0,0) circle [radius=0.6cm];
 \draw[very thick, xshift=2.6cm] (0,0) circle [radius=0.6cm];

\draw[thick, red, rounded corners=8pt, xshift=1.3cm] (-3.45,0.4) -- (-2.8,1)--(-1.3,1.4)--(0.2,1) --(0.85,0.4);
 \draw[thick, red, dashed, rounded corners=8pt, xshift=1.3cm] (-3.4,0.35) -- (-2.8,0.8)--(-1.3,1.1)--(0.2,0.8) --(0.8,0.32);
\draw[thick, red, rounded corners=4pt, xshift=-1.3cm] (3.45,-2) ..controls (3.2,-1.1) and (-0.7,-1.6).. (-0.8,-0.35);
\draw[thick, red, rounded corners=4pt, xshift=-1.3cm] (3.5,-0.4)--(3.1,-0.75) ..controls (2.2,-1.2) and (0,-1.2)..(0.8,-0.35);
\draw[thick, red, dashed, rounded corners=4pt, xshift=1.3cm ] (0.97,-0.5) -- (1.1, -0.9)--(1.1,-1.6) --(0.97,-2);
\draw[thick, red, dashed, rounded corners=4pt, xshift=-1.3cm] (-0.78,-0.27) -- (-0.4,-0.2)--(0.4,-0.2) --(0.77, -0.27); 
\node[scale=0.9] at (-1.7,1.3) {$x_1$};
\node[scale=0.9] at (-1.4,-1.5) {$y_2$};
\end{scope}

\begin{scope} [xshift=0cm, yshift=-3cm, scale=0.6]
   \draw[very thick,rounded corners=14pt] (5.23,-2) --(-4,-2)--(-4.83,0)-- (-4,2) --(5.23,2) ;
 \draw[very thick, rounded corners=10pt] (5.2,2)..controls (4.7, 1.5) and (4.7,-1.5) .. (5.2,-2) ;
 \draw[very thick,rounded corners=10pt] (5.2,2)..controls (5.7, 1.5) and (5.7,-1.5) .. (5.2,-2) ;
\draw[very thick, xshift=-2.6cm] (0,0) circle [radius=0.6cm];
 \draw[very thick, xshift=0cm] (0,0) circle [radius=0.6cm];
 \draw[very thick, xshift=2.6cm] (0,0) circle [radius=0.6cm];
\draw[thick, red, rounded corners=4pt, xshift=1.3cm] (-3.45,2) ..controls (-3.2,1.1) and (0.7,1.6).. (0.8,0.35);
\draw[thick, red, rounded corners=4pt, xshift=1.3cm] (-3.5,0.4)--(-3.1,0.75) ..controls (-2.2,1.2) and (0,1.2)..(-0.8,0.35);
\draw[thick, red, dashed, rounded corners=4pt, xshift=1.7cm ] (-3.93,0.5) -- (-4.03, 1)--(-4.03,1.6) --(-3.93,2);
\draw[thick, red, dashed, rounded corners=4pt, xshift=1.3cm] (-0.78,0.27) -- (-0.4,0.2)--(0.4,0.2) --(0.77, 0.27);
\draw[thick,  red, rounded corners=4pt, xshift=3.9cm ] (-3.93,-0.6) -- (-4.03, -1)--(-4.03,-1.6) --(-3.93,-2);
\draw[thick,  red, dashed, rounded corners=4pt, xshift=3.9cm] (-3.87,-0.6) -- (-3.77, -1)--(-3.77,-1.6) --(-3.87,-2);
\node[scale=0.9] at (1.2,1.5) {$x_2$};
\node[scale=0.9] at (0.8,-1.3) {$y_3$};
\end{scope}

\begin{scope} [xshift=8.5cm, yshift=-3cm, scale=0.6]
   \draw[very thick,rounded corners=14pt] (5.23,-2) --(-4,-2)--(-4.83,0)-- (-4,2) --(5.23,2) ;
 \draw[very thick, rounded corners=10pt] (5.2,2)..controls (4.7, 1.5) and (4.7,-1.5) .. (5.2,-2) ;
 \draw[very thick,rounded corners=10pt] (5.2,2)..controls (5.7, 1.5) and (5.7,-1.5) .. (5.2,-2) ;
  \draw[very thick, xshift=-2.6cm] (0,0) circle [radius=0.6cm];
 \draw[very thick, xshift=0cm] (0,0) circle [radius=0.6cm];
 \draw[very thick, xshift=2.6cm] (0,0) circle [radius=0.6cm];
\draw[thick, red, rounded corners=4pt, xshift=3.9cm ] (-3.93,0.6) -- (-4.03, 1)--(-4.03,1.6) --(-3.93,2);
 \draw[thick, red, dashed, rounded corners=4pt, xshift=3.9cm] (-3.87,0.6) -- (-3.77, 1)--(-3.77,1.6) --(-3.87,2);
\draw[thick, red, rounded corners=8pt, xshift=1.3cm] (-3.45,-0.4) -- (-2.8,-0.8)--(-1.3,-1.2)--(0.2,-0.8) --(0.85,-0.4);
 \draw[thick, red, dashed, rounded corners=8pt, xshift=1.3cm] (-3.4,-0.35) -- (-2.8,-0.7)--(-1.3,-1)--(0.2,-0.6) --(0.8,-0.35);

\node[scale=0.9] at (-0.7,1.3) {$x_3$};
\node[scale=0.9] at (1.6,-1.2) {$y_1$};
\end{scope}
\end{tikzpicture}
 \caption{The curves on $\Sigma_3^1$ used in Lemma~\ref{lem:g=3de3comm} and its proof.} 
 \label{fig:g=3de3comm}
\end{figure}
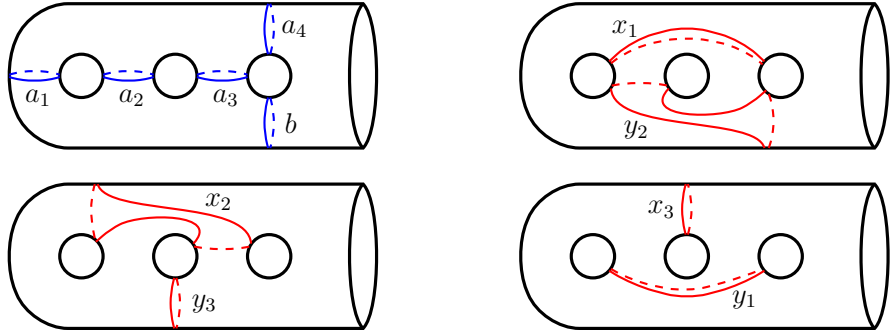
\smallskip

\smallskip
Next is a three-commutator expression we derive on the genus--$3$ surface with one boundary component. 
\begin{lemma}
\label{lem:g=3de3comm}
For $a_1, a_2, a_3, a_4$ the curves shown in $\Sigma_3^1$  in Figure~\ref{fig:g=3de3comm},  there is a relator
\[ 
t_{a_1}^{-2} t_{a_2}^{-2}t_{a_3}^{-2} t_{a_4}^{-2} C_1 C_2 C_3 =1  \ \ \ \text{ in }  \M(\Sigma_3^1)
\]
of signature $8$, where $C_1,C_2$ and $C_3$ are commutators.
\end{lemma}

\begin{proof}
By the lantern relation, we have
 \begin{equation*}
	t_{a_1}t_{a_2}t_{a_3}t_{a_4} =t_{x_2} t_{x_3} t_{x_1} 
  \end{equation*}
or 
 	\[
		t_{x_2}^{-1}t_{a_1}t_{a_2}t_{a_3}t_{a_4} =t_{x_3}t_{x_1}.
  	\] 
 This yields the equality
     \begin{eqnarray*}
	t_{x_2}^{-2}t_{a_1}^2t_{a_2}^2t_{a_3}^2t_{a_4}^2 
		&=& (t_{x_3}t_{x_1})^2\\
		&=& t_{x_3} (t_{x_3})^{t_{x_1}}   t_{x_1}^2
      \end{eqnarray*}
so that
     \begin{eqnarray*}
	t_{a_1}^2t_{a_2}^2t_{a_3}^2t_{a_4}^2 
		&=& t_{x_1}^{2} t_{x_2}^2   t_{x_3} (t_{x_3})^{t_{x_1}}. 
      \end{eqnarray*}
Hence,
  \begin{eqnarray}
t_{a_1}^4t_{a_2}^4t_{a_3}^4t_{a_4}^4 
	&=& t_{x_1}^2t_{x_2}^2t_{x_3} \left( t_{a_1}^2t_{a_2}^2t_{a_3}^2t_{a_4}^2\right) (t_{x_3})^{t_{x_1}} \nonumber \\
	&=& t_{x_1}^2t_{x_2}^2t_{x_3}  \left(  t_{x_1}^2t_{x_2}^2t_{x_3}(t_{x_3})^{t_{x_1}} \right) (t_{x_3})^{t_{x_1}} \nonumber \\
	&=& t_{x_1}^2t_{x_2}^2\left(  t_{x_1}^2t_{x_2}^2  \right)^{t_{x_3} }  t_{x_3}^2(t_{x_3}^2)^{t_{x_1}}.\label{4444=xxxx}
  \end{eqnarray}

By a similar computation, the lantern relation
   \begin{equation*}
	t_{a_1}t_{a_2}t_{a_3}t_b
		= t_{y_1}t_{y_2}t_{y_3} 
  \end{equation*}
yields
     \begin{eqnarray*}
	t_{a_1}^2t_{a_2}^2t_{a_3}^2t_{b}^2 
		&=& t_{y_2}^{2} t_{y_3}^2   t_{y_1} (t_{y_1})^{t_{y_2}} 
      \end{eqnarray*}
and
  \begin{eqnarray}
t_{a_1}^4t_{a_2}^4t_{a_3}^4t_{b}^4 
	&=& t_{y_2}^2t_{y_3}^2\left(  t_{y_2}^2t_{y_3}^2  \right)^{t_{y_1} }  t_{y_1}^2(t_{y_1}^2)^{t_{y_2}}.\label{4444=yyyy}
  \end{eqnarray}

From the equalities~\eqref{4444=xxxx} and~\eqref{4444=yyyy} we get
  \begin{eqnarray*}
t_{a_1}^8 &&  \hspace*{-0.6cm} t_{a_2}^8 t_{a_3}^8 t_{a_4}^4t_{b}^4  \\
   &=& t_{x_1}^2t_{x_2}^2\left(  t_{x_1}^2t_{x_2}^2  \right)^{t_{x_3} } 
   	 t_{x_3}^2 (t_{x_3}^2)^{t_{x_1}} 
	\cdot  t_{y_2}^2t_{y_3}^2\left(  t_{y_2}^2t_{y_3}^2  \right)^{t_{y_1} }  
	t_{y_1}^2(t_{y_1}^2)^{t_{y_2}}\\
   &=& 	\left(  t_{x_1}^2 t_{y_2}^2  t_{x_2}^2 t_{y_3}^2 \right)  
   		\left(  t_{x_1}^2 t_{y_2}^2 t_{x_2}^2  t_{y_3}^2   \right)^{X} 
		\left(   t_{x_3}^2 t_{y_1}^2  (t_{x_3}^2t_{y_1}^2)^{Y} \right),
  \end{eqnarray*}
where $X= t_{x_3}  t_{y_1}$, $Y=t_{x_1}  t_{y_2}$. We then write
 \begin{eqnarray*}
t_{a_1}^2 t_{a_2}^2t_{a_3}^2 t_{a_4}^2  
   &=&\left(  t_{x_1}^2 t_{y_2}^2 t_{a_1}^{-2} t_{a_2}^{-2} \cdot t_{x_2}^2 t_{y_3}^2 t_{a_2}^{-2} t_{a_3}^{-2} \right)  \\
   && \hspace*{1.5cm} \left(  t_{x_1}^2 t_{y_2}^2 t_{a_3}^{-2} t_{a_4}^{-2}\cdot t_{x_2}^2  t_{y_3}^2 t_{a_1}^{-2} t_{b}^{-2}  \right)^{X} \\
    &&  \hspace*{3cm} \left(   t_{x_3}^2 t_{y_1}^2t_{a_1}^{-2} t_{a_3}^{-2} \cdot  (t_{x_3}^2t_{y_1}^2t_{a_2}^{-2} t_{b}^{-2} )^{Y}\right).
  \end{eqnarray*}

We now show that each of the three factors on the right-hand side is a commutator.
It is easy to see (e.g.  by cutting the surface along the given quadruples of curves) that there are  $F_i \in \textrm{Diff}^+(\Sigma_3^1)$,  for $i=1,2,3$,  such that 
\begin{itemize}
\item  $F_1(x_1,y_2, a_1, a_2)=(a_2,a_3,x_2,y_3)$
\item  $F_2(x_1,y_2, a_3, a_4)=(a_1,b,x_2,y_3)$ 
\item  $F_3(x_3,y_1, a_1, a_3)=(a_2,b,x_3,y_1)$.
\end{itemize}

We now have
\begin{eqnarray*}
 \left( t_{x_1}^2 t_{y_2}^2 t_{a_1}^{-2}t_{a_2}^{-2}\right)  \cdot\left( t_{x_2}^2 t_{y_3}^2 t_{a_2}^{-2}t_{a_3}^{-2}\right)  
     		&=&  \left( t_{x_1}^2 t_{y_2}^2 t_{a_1}^{-2}t_{a_2}^{-2}\right)  \left( t_{a_1}^2 t_{a_2}^2 t_{x_1}^{-2}t_{y_2}^{-2}\right) ^{F_1}\\
		&=&  [t_{x_1}^2 t_{y_2}^2 t_{a_1}^{-2}t_{a_2}^{-2}, F_1],
  \end{eqnarray*}
  \begin{eqnarray*}
\left( t_{x_1}^2 t_{y_2}^2 t_{a_3}^{-2} t_{a_4}^{-2}\right) \cdot \left( t_{x_2}^2 t_{y_3}^2 t_{a_1}^{-2}t_b^{-2}\right)  
     		&=&  \left( t_{x_1}^2 t_{y_2}^2 t_{a_3}^{-2}t_{a_4}^{-2}\right)  \left( t_{a_3}^2 t_{a_4}^2 t_{x_1}^{-2}t_{y_2}^{-2}\right) ^{F_2}\\
		&=&  [t_{x_1}^2 t_{y_2}^2 t_{a_3}^{-2}t_{a_4}^{-2}, F_2],
  \end{eqnarray*}
and
 \begin{eqnarray*}
 \left( t_{x_3}^2 t_{y_1}^2  t_{a_1}^{-2}t_{a_3}^{-2}\right)   \left( t_{x_3}^2 t_{y_1}^2 t_{a_2}^{-2}t_b^{-2}\right) ^Y 
     		&=&  \left( t_{x_3}^2 t_{y_1}^2 t_{a_1}^{-2}t_{a_3}^{-2}\right) \left( t_{a_1}^2 t_{a_3}^2 t_{x_3}^{-2}t_{y_1}^{-2}\right) ^{YF_3}\\
		&=&  [t_{x_3}^2 t_{y_1}^2 t_{a_1}^{-2} t_{a_3}^{-2}, YF_3].
  \end{eqnarray*}
Since  a conjugate of a commutator is a commutator, we have written $t_{a_1}^2t_{a_2}^2t_{a_3}^2t_{a_4}^2$
as a product $C_1 C_2 C_3$ of three commutators.

A simple book keeping of the relators we have used now shows that the relator
$t_{a_1}^{-2}t_{a_2}^{-2}t_{a_3}^{-2}t_{a_4}^{-2}\cdot C_1C_2 C_3=1$
has signature $8$,  as a result of the eight lantern relators we employed in its derivation.
\end{proof}
\smallskip

\medskip
\begin{figure}[h]
\begin{tikzpicture}[scale=0.73]
\begin{scope} [xshift=0cm, yshift=0cm, scale=0.6]
  \draw[very thick, rounded corners=14pt] (-7.83,2) --(7.83,2) ;
  \draw[very thick, rounded corners=14pt] (-7.83,-2) -- (7.83,-2) ;
  \draw[very thick, xshift=2.6cm, rounded corners=10pt] (5.2,2)..controls (4.7, 1.5) and (4.7,-1.5) .. (5.2,-2) ;
  \draw[very thick, xshift=2.6cm, rounded corners=10pt] (5.2,2)..controls (5.7, 1.5) and (5.7,-1.5) .. (5.2,-2) ;
 \draw[very thick, dashed, xshift=-2.6cm, rounded corners=10pt] (-5.2,2)..controls (-4.7, 1.5) and (-4.7,-1.5) .. (-5.2,-2) ;
 \draw[very thick, xshift=-2.6cm, rounded corners=10pt] (-5.2,2)..controls (-5.7, 1.5) and (-5.7,-1.5) .. (-5.2,-2) ;
 \draw[very thick, xshift=-5.2cm] (0,0) circle [radius=0.6cm];
 \draw[very thick, xshift=-2.6cm] (0,0) circle [radius=0.6cm];
 \draw[very thick, xshift=0cm] (0,0) circle [radius=0.6cm];
 \draw[very thick, xshift=2.6cm] (0,0) circle [radius=0.6cm];
 \draw[very thick, xshift=5.2cm] (0,0) circle [radius=0.6cm];
\draw[thick,  blue, rounded corners=4pt, xshift=-5.2cm ] (-0.03,0.6) -- (-0.13, 1)--(-0.13,1.6) --(-0.03,2);
\draw[thick,  blue, dashed, rounded corners=4pt, xshift=-5.2cm ] (0.03,0.6) -- (0.13, 1)--(0.13,1.6) --(0.03,2);
  \draw[thick, blue, rounded corners=4pt, xshift=-3.9cm] (-0.7,-0.03) -- (-0.4,-0.13)--(0.4,-0.13) --(0.7,-0.03);
 \draw[thick,  blue, dashed, rounded corners=4pt, xshift=-3.9cm] (-0.7,0.03) -- (-0.4,0.13)--(0.4,0.13) --(0.7,0.03);
 \draw[thick, blue, rounded corners=4pt, xshift=-1.3cm] (-0.7,-0.03) -- (-0.4,-0.13)--(0.4,-0.13) --(0.7,-0.03);
 \draw[thick,  blue, dashed, rounded corners=4pt, xshift=-1.3cm] (-0.7,0.03) -- (-0.4,0.13)--(0.4,0.13) --(0.7,0.03);
  \draw[thick, blue, rounded corners=4pt, xshift=1.3cm] (-0.7,-0.03) -- (-0.4,-0.13)--(0.4,-0.13) --(0.7,-0.03);
 \draw[thick,  blue, dashed, rounded corners=4pt, xshift=1.3cm] (-0.7,0.03) -- (-0.4,0.13)--(0.4,0.13) --(0.7,0.03);
  \draw[thick, blue, rounded corners=4pt, xshift=3.9cm] (-0.7,-0.03) -- (-0.4,-0.13)--(0.4,-0.13) --(0.7,-0.03);
 \draw[thick,  blue, dashed, rounded corners=4pt, xshift=3.9cm] (-0.7,0.03) -- (-0.4,0.13)--(0.4,0.13) --(0.7,0.03);
\draw[thick,  blue, rounded corners=4pt, xshift=0cm ] (-0.03,0.6) -- (-0.13, 1)--(-0.13,1.6) --(-0.03,2);
\draw[thick,  blue, dashed, rounded corners=4pt, xshift=0cm ] (0.03,0.6) -- (0.13, 1)--(0.13,1.6) --(0.03,2);
\draw[thick,  blue, rounded corners=4pt, xshift=0cm ] (-0.03,-0.6) -- (-0.13, -1)--(-0.13,-1.6) --(-0.03,-2);
\draw[thick,  blue, dashed, rounded corners=4pt, xshift=0cm ] (0.03,-0.6) -- (0.13, -1)--(0.13,-1.6) --(0.03,-2);
\draw[thick,  blue, rounded corners=4pt, xshift=5.2cm ] (-0.03,-0.6) -- (-0.13, -1)--(-0.13,-1.6) --(-0.03,-2);
\draw[thick,  blue, dashed, rounded corners=4pt, xshift=5.2cm ] (0.03,-0.6) -- (0.13, -1)--(0.13,-1.6) --(0.03,-2);
\node[scale=0.9] at (-5.8,1.3) {$a_1$};
\node[scale=0.9] at (-3.8,-0.6) {$a_2$};
\node[scale=0.9] at (-1.2,-0.6) {$a_3$};
\node[scale=0.9] at (1.4,-0.6) {$a_4$};
\node[scale=0.9] at (4,-0.6) {$a_5$};
\node[scale=0.9] at (0.5,1.3) {$a$};
\node[scale=0.9] at (0.5,-1.3) {$b$};
\node[scale=0.9] at (5.7,-1.3) {$c$};
\end{scope}
\begin{scope} [xshift=10cm, yshift=0cm, scale=0.6]
  \draw[very thick, rounded corners=14pt] (-7.83,2) --(7.83,2) ;
  \draw[very thick, rounded corners=14pt] (-7.83,-2) -- (7.83,-2) ;
  \draw[very thick, xshift=2.6cm, rounded corners=10pt] (5.2,2)..controls (4.7, 1.5) and (4.7,-1.5) .. (5.2,-2) ;
  \draw[very thick, xshift=2.6cm, rounded corners=10pt] (5.2,2)..controls (5.7, 1.5) and (5.7,-1.5) .. (5.2,-2) ;
 \draw[very thick, dashed, xshift=-2.6cm, rounded corners=10pt] (-5.2,2)..controls (-4.7, 1.5) and (-4.7,-1.5) .. (-5.2,-2) ;
 \draw[very thick, xshift=-2.6cm, rounded corners=10pt] (-5.2,2)..controls (-5.7, 1.5) and (-5.7,-1.5) .. (-5.2,-2) ;
 \draw[very thick, xshift=-5.2cm] (0,0) circle [radius=0.6cm];
 \draw[very thick, xshift=-2.6cm] (0,0) circle [radius=0.6cm];
 \draw[very thick, xshift=0cm] (0,0) circle [radius=0.6cm];
 \draw[very thick, xshift=2.6cm] (0,0) circle [radius=0.6cm];
 \draw[very thick, xshift=5.2cm] (0,0) circle [radius=0.6cm];
\draw[thick, red, rounded corners=8pt, xshift=-1.3cm] (-3.45,0.4) -- (-2.8,1)--(-1.3,1.4)--(0.2,1) --(0.85,0.4);
 \draw[thick, red, dashed, rounded corners=8pt, xshift=-1.3cm] (-3.4,0.35) -- (-2.8,0.8)--(-1.3,1.1)--(0.2,0.8) --(0.8,0.32);
\draw[thick, red, rounded corners=4pt, xshift=1.3cm] (3.45,-2) ..controls (3.2,-1.1) and (-0.7,-1.6).. (-0.8,-0.35);
\draw[thick, red, rounded corners=4pt, xshift=1.3cm] (3.5,-0.4)--(3.1,-0.75) ..controls (2.2,-1.2) and (0,-1.2)..(0.8,-0.35);
\draw[thick, red, dashed, rounded corners=4pt, xshift=3.9cm ] (0.97,-0.5) -- (1.1, -0.9)--(1.1,-1.6) --(0.97,-2);
\draw[thick, red, dashed, rounded corners=4pt, xshift=1.3cm] (-0.78,-0.27) -- (-0.4,-0.2)--(0.4,-0.2) --(0.77, -0.27); 
  \draw[thick, blue, rounded corners=4pt, xshift=-3.9cm] (-0.7,-0.03) -- (-0.4,-0.13)--(0.4,-0.13) --(0.7,-0.03);
 \draw[thick,  blue, dashed, rounded corners=4pt, xshift=-3.9cm] (-0.7,0.03) -- (-0.4,0.13)--(0.4,0.13) --(0.7,0.03);
  \draw[thick, blue, rounded corners=4pt, xshift=1.3cm] (-0.7,-0.03) -- (-0.4,-0.13)--(0.4,-0.13) --(0.7,-0.03);
 \draw[thick,  blue, dashed, rounded corners=4pt, xshift=1.3cm] (-0.7,0.03) -- (-0.4,0.13)--(0.4,0.13) --(0.7,0.03);
\node[scale=0.9] at (-4.3,1.3) {$x_1$};
\node[scale=0.9] at (1.2,-1.5) {$y_2$};
\end{scope}
\begin{scope} [xshift=0cm, yshift=-3cm, scale=0.6]
  \draw[very thick, rounded corners=14pt] (-7.83,2) --(7.83,2) ;
  \draw[very thick, rounded corners=14pt] (-7.83,-2) -- (7.83,-2) ;
  \draw[very thick, xshift=2.6cm, rounded corners=10pt] (5.2,2)..controls (4.7, 1.5) and (4.7,-1.5) .. (5.2,-2) ;
  \draw[very thick, xshift=2.6cm, rounded corners=10pt] (5.2,2)..controls (5.7, 1.5) and (5.7,-1.5) .. (5.2,-2) ;
 \draw[very thick, dashed, xshift=-2.6cm, rounded corners=10pt] (-5.2,2)..controls (-4.7, 1.5) and (-4.7,-1.5) .. (-5.2,-2) ;
 \draw[very thick, xshift=-2.6cm, rounded corners=10pt] (-5.2,2)..controls (-5.7, 1.5) and (-5.7,-1.5) .. (-5.2,-2) ;
 \draw[very thick, xshift=-5.2cm] (0,0) circle [radius=0.6cm];
 \draw[very thick, xshift=-2.6cm] (0,0) circle [radius=0.6cm];
 \draw[very thick, xshift=0cm] (0,0) circle [radius=0.6cm];
 \draw[very thick, xshift=2.6cm] (0,0) circle [radius=0.6cm];
 \draw[very thick, xshift=5.2cm] (0,0) circle [radius=0.6cm];
\draw[thick, red, rounded corners=4pt, xshift=-1.3cm] (-3.45,2) ..controls (-3.2,1.1) and (0.7,1.6).. (0.8,0.35);
\draw[thick, red, rounded corners=4pt, xshift=-1.3cm] (-3.5,0.4)--(-3.1,0.75) ..controls (-2.2,1.2) and (0,1.2)..(-0.8,0.35);
\draw[thick, red, dashed, rounded corners=4pt, xshift=-0.9cm ] (-3.93,0.5) -- (-4.03, 1)--(-4.03,1.6) --(-3.93,2);
\draw[thick, red, dashed, rounded corners=4pt, xshift=-1.3cm] (-0.78,0.27) -- (-0.4,0.2)--(0.4,0.2) --(0.77, 0.27);
\draw[thick,  red, rounded corners=4pt, xshift=6.5cm ] (-3.93,-0.6) -- (-4.03, -1)--(-4.03,-1.6) --(-3.93,-2);
\draw[thick,  red, dashed, rounded corners=4pt, xshift=6.5cm] (-3.87,-0.6) -- (-3.77, -1)--(-3.77,-1.6) --(-3.87,-2);
 \draw[thick, blue, rounded corners=4pt, xshift=-1.3cm] (-0.7,-0.03) -- (-0.4,-0.13)--(0.4,-0.13) --(0.7,-0.03);
 \draw[thick,  blue, dashed, rounded corners=4pt, xshift=-1.3cm] (-0.7,0.03) -- (-0.4,0.13)--(0.4,0.13) --(0.7,0.03);
\draw[thick,  blue, rounded corners=4pt, xshift=5.2cm ] (-0.03,-0.6) -- (-0.13, -1)--(-0.13,-1.6) --(-0.03,-2);
\draw[thick,  blue, dashed, rounded corners=4pt, xshift=5.2cm ] (0.03,-0.6) -- (0.13, -1)--(0.13,-1.6) --(0.03,-2);
\node[scale=0.9] at (-1.4,1.5) {$x_2$};
\node[scale=0.9] at (3.3,-1.3) {$y_3$};
\end{scope}
\begin{scope} [xshift=10cm, yshift=-3cm, scale=0.6]
  \draw[very thick, rounded corners=14pt] (-7.83,2) --(7.83,2) ;
  \draw[very thick, rounded corners=14pt] (-7.83,-2) -- (7.83,-2) ;
  \draw[very thick, xshift=2.6cm, rounded corners=10pt] (5.2,2)..controls (4.7, 1.5) and (4.7,-1.5) .. (5.2,-2) ;
  \draw[very thick, xshift=2.6cm, rounded corners=10pt] (5.2,2)..controls (5.7, 1.5) and (5.7,-1.5) .. (5.2,-2) ;
 \draw[very thick, dashed, xshift=-2.6cm, rounded corners=10pt] (-5.2,2)..controls (-4.7, 1.5) and (-4.7,-1.5) .. (-5.2,-2) ;
 \draw[very thick, xshift=-2.6cm, rounded corners=10pt] (-5.2,2)..controls (-5.7, 1.5) and (-5.7,-1.5) .. (-5.2,-2) ;
 \draw[very thick, xshift=-5.2cm] (0,0) circle [radius=0.6cm];
 \draw[very thick, xshift=-2.6cm] (0,0) circle [radius=0.6cm];
 \draw[very thick, xshift=0cm] (0,0) circle [radius=0.6cm];
 \draw[very thick, xshift=2.6cm] (0,0) circle [radius=0.6cm];
 \draw[very thick, xshift=5.2cm] (0,0) circle [radius=0.6cm];
\draw[thick, red, rounded corners=4pt, xshift=1.3cm ] (-3.93,0.6) -- (-4.03, 1)--(-4.03,1.6) --(-3.93,2);
 \draw[thick, red, dashed, rounded corners=4pt, xshift=1.3cm] (-3.87,0.6) -- (-3.77, 1)--(-3.77,1.6) --(-3.87,2);
\draw[thick, red, rounded corners=8pt, xshift=3.9cm] (-3.45,-0.4) -- (-2.8,-0.8)--(-1.3,-1.2)--(0.2,-0.8) --(0.85,-0.4);
 \draw[thick, red, dashed, rounded corners=8pt, xshift=3.9cm] (-3.4,-0.35) -- (-2.8,-0.7)--(-1.3,-1)--(0.2,-0.6) --(0.8,-0.35);
\node[scale=0.9] at (-3.3,1.3) {$x_3$};
\node[scale=0.9] at (4.2,-1.2) {$y_1$};
\end{scope}
\end{tikzpicture}
\caption{The curves on $\Sigma_5^2$ used in the proof of Lemma~\ref{lem:g=5}. }
\label{fig:g=5de2comm}
\end{figure}
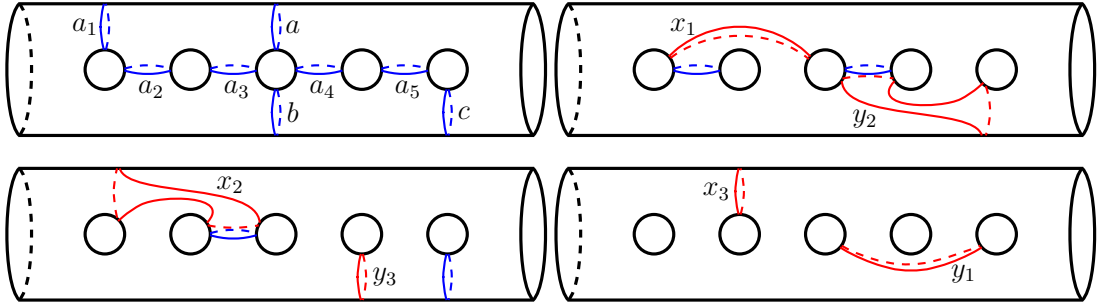

\bigskip
Lastly, we obtain a two-commutator expression on the genus--$5$ surface $\Sigma_5^2$:

\begin{lemma}\label{lem:g=5}
Let $a$ and $b$ be pairwise disjoint nonseparating curves on $\Sigma_5^2$ such that $a\cup b$
separates the surface $\Sigma_5^2$ into two genus--$2$ components, and let $x$ be another curve nonseparating
in the complement of $a\cup b$.  Then there is a relator
\[
 t_{a}^{-1} t_{b}^{-1}t_{x}^{-2}C_1C_2=1  \ \ \ \text{ in }  \M(\Sigma_5^2)
 \]
 with signature $4$.
 \end{lemma}
\begin{proof}
Consider the curves on $\Sigma_5^2$ as illustrated in Figure~\ref{fig:g=5de2comm}.
By the lantern relation, we have
 \begin{eqnarray*}
	t_{a_1}t_{a_2}t_{a_3}t_{a} &=& t_{x_2}t_{x_3}t_{x_1},\\ 
	t_{b}t_{a_4}t_{a_5}t_{c} &=& t_{y_1}t_{y_2}t_{y_3}. 
  \end{eqnarray*}
It follows now as in the proof of Lemma~\ref{lem:g=3de3comm} that
 \begin{eqnarray*}
	t_{a_1}^2t_{a_2}^2t_{a_3}^2t_{a}^2 
		&=& t_{x_1}^2t_{x_2}^2 t_{x_3} t_{x_3}^{t_{x_1}}, \\
	t_{b}^2t_{a_4}^2t_{a_5}^2t_{c}^2 
		&=& t_{y_2}^2t_{y_3}^2 t_{y_1} t_{y_1}^{t_{y_2}},
 \end{eqnarray*}
so that we have
 \begin{eqnarray*}
	t_{a_1}^2t_{a_2}^2t_{a_3}^2t_{a_4}^2t_{a_5}^2t_{a}^2 t_{b}^2t_{c}^2 
		&=& t_{x_1}^2 t_{y_2}^2 \cdot t_{x_2}^2 t_{y_3}^2 \cdot t_{x_3}  t_{y_1}  \cdot   (t_{x_3}t_{y_1})^{t_{x_1}t_{y_2}},
 \end{eqnarray*}
or equivalently
\[
t_at_b t_{a_1}^2
=
  (\, t_{x_1}^2 t_{y_2}^2 t_{a_2}^{-2}t_{a_4}^{-2}  \cdot \, t_{x_2}^2 t_{y_3}^2  t_{a_3}^{-2} t_{c}^{-2} \,)
  \cdot
   (\, t_{x_3}  t_{y_1} t_a^{-1} t_{a_5}^{-1}  \cdot   (t_{x_3} t_{y_1}t_b^{-1}t_{a_5}^{-1})^{t_{x_1}t_{y_2}} \,).   
\]
One can easily find $F_1, F_2 \in \textrm{Diff}^+(\Sigma_5^2)$ such that
\begin{itemize}
\item  $F_1(x_1,y_2,a_2,a_4)=(a_3,c,x_2,y_3)$,
\item $F_2(x_3,y_1,a,a_5)=(b,a_5,x_3,y_1)$.
\end{itemize}
It follows that,  by using~\eqref{1comm} once again, we can express  $t_{x_1}^2 t_{y_2}^2 t_{a_2}^{-2}t_{a_4}^{-2}  \cdot \, t_{x_2}^2 t_{y_3}^2  t_{a_3}^{-2} t_{c}^{-2}$,  as well as $t_{x_3}  t_{y_1} t_a^{-1} t_{a_5}^{-1}  \cdot   (t_{x_3} t_{y_1}t_b^{-1}t_{a_5}^{-1})^{t_{x_1}t_{y_2}}$,  as a single commutator.   
We thus get a relator
$ t_{a}^{-1} t_{b}^{-1}t_{a_1}^{-2}C_1C_2=1$,  with signature $4$.
\end{proof}

\medskip
\subsection{Commutator lengths of some mapping classes}  \

 For an element $x$ in the commutator subgroup of a group $G$, let $\cl(x)$ denote its \emph{commutator length}, the minimum number of commutators needed to 
 express $x$ as a product of commutators, and let  $ \scl (x):=\lim_{n\to\infty} \frac{\cl (x^n)}{n}$ 
 be its \emph{stable commutator length} \cite{Calegari}.

 Let $c$ be a nonseparating curve on $\Sigma_g^b$.
 Recall that the mapping class group $\M(\Sigma_g^b)$ is perfect for $g \geq 3$, whereas $H_1(\M(\Sigma_2^b))=\Z_{10}$, generated by the class of $t_c$; see e.g.~\cite{Korkmaz2002}.  It should be clear from our rendition of 
 Tsuboi's trick in Section~\ref{sec:comm4comm} that $\cl(t_c^n)=2$ in $\M(\Sigma_g^b)$ 
 when $g$ is large enough (c.f. \cite{KotschickInfinite}). On the other hand, $\cl (t_c^n) \geq 2$ for any $g, n \in \Z^+$ \cite{BraungardtKotschick}. 
 Ensuing these facts is the question below,  which is a refinement of Mess' question on 
 $\cl (t_c^n)$ in \cite{Kirby}. To ease the upcoming discussion, let us define $g_2(n)$ to be 
 the minimum genus $g \geq 2$ such that $\cl(t_c^n)=2$ in $\M(\Sigma_g^1)$ 
 for a nonseparating curve $c$, where it makes sense only to consider $n \equiv 0$ (mod $10$) when $g=2$.
 
\begin{question}
For a given positive integer $n$,  what is $g_2(n)$?
\end{question}

 The minimal genus in $\M(\Sigma_g^b)$ with $b>1$ is the same as the one in $\M(\Sigma_g^1)$
 since two surfaces of 
 the same genus with boundaries embed into each other,  and thus, we can cater commutator 
 expressions from one another.  By the same reasoning,  any commutator expression for 
 $t_c^n$ in $\M(\Sigma_g^1)$ provides an upper bound for the minimal genus in $\M(\Sigma_g)$, too. However, it is not immediately clear to us whether the qualitative difference here would translate to a quantitative one; in fact, when $g=1$, we can show that $\cl(t_c^{12})=2$ in $\M(\Sigma_1)$, but certainly not in $\M(\Sigma_1^1)$, where no nontrivial power of $t_c$ is a product of commutators. Lastly, let us add that there is no reason ---other than us not wanting to digress here any further---
 for not considering this question for separating curves.
 
 Without sharper lower bounds in hand,  it is  challenging to answer the above question 
 in full generality.  Nonetheless we are able to determine the minimal genus $g_2(n)$ for a few values of  $n$. We record them here:

\begin{cor} \label{cor:cl}
We have $g_2(1)=g_2(2)=g_2(4)=3$, and $g_2(10)=2$.
\end{cor}

\begin{proof}
It was shown by Ozbagci and the second author in \cite{KorkmazOzbagci} that $\cl(t_c)=2$ when $g =3$, and by the second author in \cite{Korkmaz} that $\cl(t_c^{10})=2$ when $g = 2$. 

For $t_c^2$ and $t_c^4$ we proceed as follows. 
Let $\{\delta_1,\delta_2,\delta_3\}$ denote the three components of $\partial \Sigma_2^3$.  By Lemma~\ref{lem:genus=2},  we have $t_{\delta_1} t_{\delta_2} = C_1 C_2$ as well as $t_{\delta_1}^2 t_{\delta_2}^2 = D_1 D_2$, for some commutators $C_j, D_j$ in $\M(\Sigma_2^3)$.  Consider an embedding $\Sigma_2^3 \hookrightarrow \Sigma_3^1$ obtained by attaching a cylinder to the two boundary components $\delta_1$ and $\delta_2$ of $\Sigma_2^3$.  The image of $\delta_1$ and $\delta_2$ are  isotopic to the same non-separating curve $c$ in $\Sigma_3^1$, whereas the remaining boundary component maps to the unique boundary component of $\Sigma_3^1$.  Using the homomorphism $\M(\Sigma_2^3) \to \M(\Sigma_3^1)$ induced by this embedding,  we thus derive two new expressions  of the form $t_c^2=C'_1 C'_2$ and $t_c^4 = D'_1 D'_2$ in $\M(\Sigma_3^1)$, where $C'_j,  D'_j$ are commutators.   

None of the $t_c^n$ considered above are in the commutator subgroup of $\M(\Sigma_g^1)$ for smaller $g$, so we get the claimed values for $g_2(n)$.
\end{proof}

\begin{remark}
Corollary~\ref{cor:cl} provides a complete (meaning,  for all $g \geq 3$) answer to Problem 2.13(b) in Kirby's  List \cite{Kirby} for $n=1, 2$ and $4$.  Using similar arguments,  we can also conclude that $\textrm{cl}(t_c^3) \leq 3$ for all $g \geq 3$ and is equal to $2$ for $g \geq 5$. 
\end{remark}

\begin{remark} 
  If $t_c^n$ is expressed as a product of two commutators in $\M(\Sigma_g^1)$, 
  we get a genus--$g$ Lefschetz fibration over $\Sigma_2$ with $n$ nodes 
  clustered all in one fiber. Then~\cite[Theorem~8] {BraungardtKotschick} dictates 
  that $g\geq \frac{n+6}{18}$. Thus, we have $g_2(n) \geq \frac{n+6}{18}$ for every $n$. 
\end{remark}

\medskip
Lastly,  we look at the stable commutator length of the boundary multitwist \newline $\Delta = t_{\delta_1} 
 t_{\delta_2} \cdots t_{\delta_b}$ in  $\M(\Sigma_g^b)$,  for $g \geq 2$,  $b>0$.  Note that $\Delta$ is in the commutator subgroup of $\M(\Sigma_g^b)$, even when $g =2$. 

In~\cite{BaykurKorkmazMonden},  it was shown by Monden and the authors of 
this article that when $b=1$,  $\cl(\Delta^n)=\lfloor (n+3)/ 2 \rfloor$ 
for any positive $n$, so that $\scl(\Delta)=1/2$.  
The main gain in the case of a boundary multitwist is the sharp lower bounds we get,  
which can be interpreted as a manifestation of 
the Milnor-Wood inequality  \cite{BaykurFlatBundles, HamenstadtMW}.  
As the same lower bound carries over to $\M(\Sigma_g^b)$ for any $b>1$,   
precise calculations of $\cl(\Delta^n)$ and $\scl(\Delta)$ are possible 
any time we are able to realize the lower bound.  And when we cap the extra boundary component,  Lemma~\ref{lem:genus=2} does precisely this for $b=2$,  as we can now 
express $\Delta^n=t_{\delta_1}^n t_{\delta_2}^n$ as a product of $\lfloor (n+3)/2 \rfloor$ 
commutators in $\M(\Sigma_2^2)$.  In summary,  when $b=2$, we also have 
$\cl(\Delta^n)=\lfloor (n+3)/2 \rfloor$ and $\scl(\Delta)=1/2$ in $\M(\Sigma_g^2)$.  
We record these calculations as well:

\begin{cor}\label{cor:scl}
Let $\Delta$ be the boundary multitwist in $\M(\Sigma_g^b)$.  For any $g \geq 2$ and $b=1, 2$,  we have $\scl(\Delta)=1/2$.
\end{cor}

There is a little more we can say here:  Let us say that a sequence $(c_1, c_2,\ldots, c_{k})$ of curves on 
 a surface is a \textit{chain} if $c_i$ and $c_{j}$ intersect transversely at one point for $j=i\pm 1$ and 
 are disjoint otherwise.  For a chain $(c_1,  c_2,\ldots, c_{2g})$ on $\Sigma_g^1$, 
 and a chain $(c_1,  c_2,\ldots, c_{2g+1})$ on $\Sigma_g^2$, consider the elements 
\[  
S:=t_{c_1}t_{c_2}\cdots t_{c_{2g}} \in \M(\Sigma_g^1)  \ \ \text{ and } \ 
T:=t_{c_1}t_{c_2}\cdots t_{c_{2g+1}} \in \M(\Sigma_g^2).
\]
 It is well-known that  $S^{4g+2}=\Delta$ and $T^{2g+2}=\Delta$; see e.g.\,\cite{FarbMargalit}. Since $\scl$ is homogeneous and since $\scl(\Delta)=1/2$,  we  get $\scl(S)=1/(4(2g+1))$ in $\M(\Sigma_g^1)$ and $\scl(T)=1/(4(g+1))$ in $\M(\Sigma_g^2)$, $g \geq 2$.

\medskip
\section{New surface bundles with positive signatures} \label{sec:final}

This final section is dedicated to the proof of our main theorem.

\noindent \underline{\textit{Proof of Theorem~\ref{main}}}:
 We begin with an elementary observation (cf. \cite{EKKOS}): Say  we have a monodromy factorization  
 \[ C_1 \cdots C_h=1 \ \ \ \text{ in }  \M(\Sigma_g^1)
 \]
 for a $\Sigma_g$--bundle over $\Sigma_h$ with a section of self-intersection zero, where $\{C_i\}$ 
 are  commutators.  Call this surface bundle $(X,f)$.  Then for any given $g' \geq g$ and $h' \geq h$, 
 we can derive another monodromy factorization
 \[ C'_1 \cdots C'_h C'_{h+1} \cdots C_{h'} =1 \ \ \ \text{ in }  \M(\Sigma_{g'}^1)
 \]
 for a $\Sigma_{g'}$--bundle over $\Sigma_{h'}$ with a section of self-intersection zero,  by using 
 any embedding $\Sigma_g^1 \hookrightarrow \Sigma_{g'}^1$ and concatenating 
 the trivial commutators $C'_{h+1}, \ldots, C'_h$ to the product of the commutators $C'_i$ 
 which are the  images of the original commutators $C_i$.  Since embedding a relation or 
 adding trivial commutators do not change the signature,  for the new surface bundle  $(X',f')$ we obtained,  we have $\sigma(X')=\sigma(X)$.\footnote{
 This  construction is good enough to address the mere existence of positive signature surface bundles with prescribed fiber and base genera.  Otherwise,  to generate surface bundles with larger signatures relative to their topology,  it is certainly better to use extensions with nontrivial surface bundles with positive signatures. }

It should be noted that  in order to increase the fiber genus here,  we needed the initial surface bundle $(X,f)$  to have a section of self-intersection zero,  or equivalently,  a commutator identity supported on $\Sigma_g^1$ as opposed to $\Sigma_g$.  Recall that this is also a necessary condition to invoke Theorem~\ref{thm:tsuboish}.

With the above in mind,  we will  generate the promised surface bundles with positive signatures from a few base examples.

\smallskip
\noindent \underline{\textit{$h \geq 5$ and $g \geq 3$\,}}:  For the curves $a_i, \delta_j$ on $\Sigma_1^4$ as shown on left in Figure~\ref{fig:(3,5)proof},  the \emph{four-holed torus relation} in \cite{KorkmazOzbagci2} gives us a relator 
 \[
	t_{\delta_1}^{-1}t_{\delta_2}^{-1}t_{\delta_3}^{-1}t_{\delta_4}^{-1}
	(t_{c_0}t_{c_1}t_{c_3}t_{c_0}t_{c_2}t_{c_4})^2=1
\ \ \ \ \text{in } \M(\Sigma_1^4) \, 
 \]
with $\sigma= -4$.  We embed $\Sigma_1^4 \hookrightarrow \Sigma_3^1$ so that the images of the boundary components $\delta_i$ of $\Sigma_1^4$ are as shown on the right hand side of Figure~\ref{fig:(3,5)proof},  denoted by the same letters.  As  $\delta_2$ and $\delta_4$ become isotopic in $\Sigma_3^1$, we represent them by the same curve after the embedding.  The image of each $c_i$ is labeled as $a_i$.

\medskip
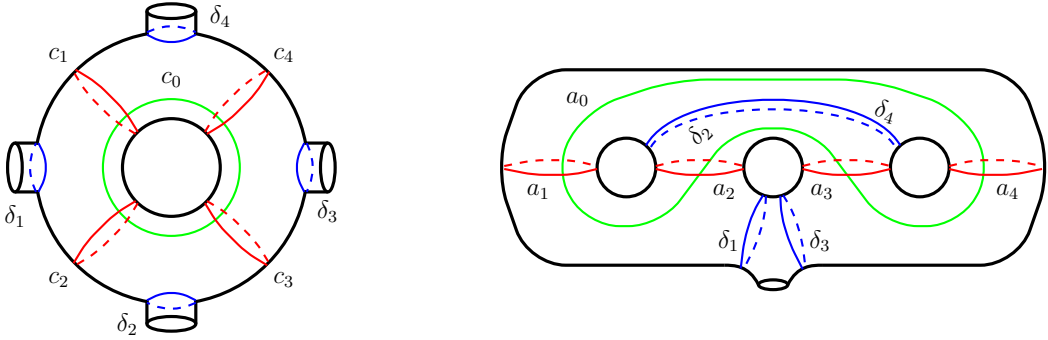
\begin{figure}[h]
\begin{tikzpicture}[scale=1]
\begin{scope}  [xshift=0cm, yshift=0cm, scale=0.65] 
   \draw[very thick, xshift=0cm] (0,0) circle [radius=1cm];
    \draw[very thick ] (-2.75,0.467) arc (170:99.7:2.8);
  \draw[very thick ] (2.75,0.47) arc (10:80.3:2.8);
  \draw[very thick, rotate=90 ] (-2.75,0.467) arc (170:99.7:2.8);
 \draw[very thick, rotate=-90 ] (2.75,0.47) arc (10:80.3:2.8);
  \draw[very thick, rounded corners=10pt] (-0.5, 2.73)--(-0.5 ,3.2);
   \draw[very thick, rounded corners=10pt] (0.5, 2.73)--(0.5 ,3.2);
   \draw[very thick, rounded corners=10pt, rotate=90] (-0.5, 2.73)--(-0.5 ,3.2);
   \draw[very thick, rounded corners=10pt, rotate=90] (0.5, 2.73)--(0.5 ,3.2);
   \draw[very thick, rounded corners=10pt, rotate=180] (-0.5, 2.73)--(-0.5 ,3.2);
   \draw[very thick, rounded corners=10pt, rotate=180] (0.5, 2.73)--(0.5 ,3.2);
   \draw[very thick, rounded corners=10pt, rotate=270] (-0.5, 2.73)--(-0.5 ,3.2);
   \draw[very thick, rounded corners=10pt, rotate=270] (0.5, 2.73)--(0.5 ,3.2);
 \draw[very thick, rotate=0] (0,3.2) ellipse (0.5 and 0.15);
 \draw[very thick, rotate=90] (0,3.2) ellipse (0.5 and 0.15);
 \draw[very thick, rotate=180] (0,3.2) ellipse (0.5 and 0.15);
 \draw[very thick, rotate=270] (0,3.2) ellipse (0.5 and 0.15);
  \draw[thick, blue, rounded corners=3pt, yshift=-8, xshift=0cm, rotate=0] (-0.5,3) -- (-0.2,2.85)--(0.2,2.85) --(0.5,3);
  \draw[thick,  blue, dashed, rounded corners=3pt, yshift=-8, xshift=0cm, rotate=0] (-0.5,3.05) -- (-0.2,3.17)--(0.2,3.17) --(0.5,3.05);
  \draw[thick, blue, rounded corners=3pt, xshift=8, rotate=90] (-0.5,3) -- (-0.2,2.85)--(0.2,2.85) --(0.5,3);
  \draw[thick,  blue, dashed, rounded corners=3pt, xshift=8, rotate=90] (-0.5,3.05) -- (-0.2,3.17)--(0.2,3.17) --(0.5,3.05);
  \draw[thick, blue, rounded corners=3pt, yshift=8, xshift=0cm, rotate=180] (-0.5,3) -- (-0.2,2.85)--(0.2,2.85) --(0.5,3);
  \draw[thick,  blue, dashed, rounded corners=3pt, yshift=8, xshift=0cm, rotate=180] (-0.5,3.05) -- (-0.2,3.17)--(0.2,3.17) --(0.5,3.05);
 \draw[thick, blue, rounded corners=3pt, xshift=-8, rotate=270] (-0.5,3) -- (-0.2,2.85)--(0.2,2.85) --(0.5,3);
  \draw[thick,  blue, dashed, rounded corners=3pt, xshift=-8, rotate=270] (-0.5,3.05) -- (-0.2,3.17)--(0.2,3.17) --(0.5,3.05);
  
  \draw[thick, green, xshift=0cm] (0,0) circle [radius=1.4cm];
   \draw[thick, red, rounded corners=4pt, xshift=0cm, rotate=45] (1,-0.03) -- (1.5,-0.15)--(2.3,-0.15) --(2.8,-0.03);
   \draw[thick,  red, dashed, rounded corners=4pt, xshift=0cm, rotate=45] (1,0.03) -- (1.5,0.15)--(2.3,0.15) --(2.8,0.03);
    \draw[thick, red, rounded corners=4pt, xshift=0cm, rotate=135] (1,-0.03) -- (1.5,-0.15)--(2.3,-0.15) --(2.8,-0.03);
   \draw[thick,  red, dashed, rounded corners=4pt, xshift=0cm, rotate=135] (1,0.03) -- (1.5,0.15)--(2.3,0.15) --(2.8,0.03);
     \draw[thick, red, rounded corners=4pt, xshift=0cm, rotate=-45] (1,-0.03) -- (1.5,-0.15)--(2.3,-0.15) --(2.8,-0.03);
   \draw[thick,  red, dashed, rounded corners=4pt, xshift=0cm, rotate=-45] (1,0.03) -- (1.5,0.15)--(2.3,0.15) --(2.8,0.03);
    \draw[thick, red, rounded corners=4pt, xshift=0cm, rotate=-135] (1,-0.03) -- (1.5,-0.15)--(2.3,-0.15) --(2.8,-0.03);
   \draw[thick,  red, dashed, rounded corners=4pt, xshift=0cm, rotate=-135] (1,0.03) -- (1.5,0.15)--(2.3,0.15) --(2.8,0.03);
 \node[scale=0.8] at (0,1.8) {$c_0$};
 \node[scale=0.8] at (-2.3,2.3) {$c_1$};
 \node[scale=0.8] at (-2.3,-2.3) {$c_2$};
 \node[scale=0.8] at (2.3,2.3) {$c_4$};
 \node[scale=0.8] at (2.3,-2.3) {$c_3$};
  \node[scale=0.8] at (-3.2,-1) {$\delta_1$};
 \node[scale=0.8] at (-0.9,-3.2) {$\delta_2$};
 \node[scale=0.8] at (3.2,-0.9) {$\delta_3$};
 \node[scale=0.8] at (1,3.1) {$\delta_4$};
\end{scope}


  \begin{scope} [xshift=8cm, yshift=0cm, scale=0.65]
  	\draw[very thick, rounded corners=14pt] (-1,-2)--(-5,-2) -- (-5.7,0)--(-5, 2) 
		-- (5,2)--(5.7,0) -- (5,-2)--(1,-2);
  	\draw[very thick, rounded corners=5pt] (-1,-2)..controls (-0.8, -2) and (-0.5,-2) .. (-0.3,-2.4) ;
  	\draw[very thick, rounded corners=5pt] (1,-2)..controls (0.8, -2) and (0.5,-2) .. (0.3,-2.4) ;
  	\draw[very thick] (0,-2.4) ellipse (0.3 cm and 0.1 cm);

  \draw[thick, green, rounded corners=12pt] (-0.8,0.8)-- (-2.3,-1.2)--(-3.8,-1.2)-- (-4.5,0)--(-3.8,1.2)--
  (-2,1.8)--(2,1.8)--(3.8,1.2)--(4.5,0)--(3.8, -1.2)-- (2.3,-1.2)--(0.8,0.8)--cycle;

 \draw[very thick, xshift=-3cm] (0,0) circle [radius=0.6cm];
 \draw[very thick, xshift=0cm] (0,0) circle [radius=0.6cm];
 \draw[very thick, xshift=3cm] (0,0) circle [radius=0.6cm];
  \draw[thick, red, rounded corners=4pt, xshift=-4.5cm] (-1,-0.03) -- (-0.4,-0.15)--(0.4,-0.15) --(0.9,-0.03);
 \draw[thick,  red, dashed, rounded corners=4pt, xshift=-4.5cm] (-1,0.03) -- (-0.4,0.15)--(0.4,0.15) --(0.9,0.03);
  \draw[thick, red, rounded corners=4pt, xshift=-1.5cm] (-0.9,-0.03) -- (-0.4,-0.15)--(0.4,-0.15) --(0.9,-0.03);
 \draw[thick,  red, dashed, rounded corners=4pt, xshift=-1.5cm] (-0.9,0.03) -- (-0.4,0.15)--(0.4,0.15) --(0.9,0.03);
  \draw[thick, red, rounded corners=4pt, xshift=1.5cm] (-0.9,-0.03) -- (-0.4,-0.15)--(0.4,-0.15) --(0.9,-0.03);
 \draw[thick,  red, dashed, rounded corners=4pt, xshift=1.5cm] (-0.9,0.03) -- (-0.4,0.15)--(0.4,0.15) --(0.9,0.03);
  \draw[thick, red, rounded corners=4pt, xshift=4.5cm] (-0.9,-0.03) -- (-0.4,-0.15)--(0.4,-0.15) --(1,-0.03);
 \draw[thick,  red, dashed, rounded corners=4pt, xshift=4.5cm] (-0.9,0.03) -- (-0.4,0.15)--(0.4,0.15) --(1,0.03);
\draw[thick,  blue, rounded corners=4pt, xshift=0cm, rotate=-17 ] (-0.03,-0.6) -- (-0.13, -1)--(-0.13,-1.6) --(-0.03,-2.17);
\draw[thick,  blue, dashed, rounded corners=4pt, xshift=0cm, rotate=-17] (0.03,-0.6) -- (0.13, -1)--(0.13,-1.6) --(0.03,-2.17);
\draw[thick,  blue, rounded corners=4pt, xshift=0cm, rotate=17 ] (-0.03,-0.6) -- (-0.13, -1)--(-0.13,-1.6) --(-0.03,-2.17);
\draw[thick,  blue, dashed, rounded corners=4pt, xshift=0cm, rotate=17] (0.03,-0.6) -- (0.13, -1)--(0.13,-1.6) --(0.03,-2.17);
  \draw[thick, blue, rounded corners=8pt] (-2.6,0.45)..controls (-1.9, 1.7) and (1.9,1.7) ..(2.6,0.45);
  \draw[thick, blue, dashed, rounded corners=8pt] (-2.5,0.42)..controls (-1.9, 1.45) and (1.9,1.45) ..(2.5,0.42);
\node[scale=0.8] at (-4,1.4) {$a_0$};
\node[scale=0.8] at (-4.8,-0.5) {$a_1$};
\node[scale=0.8] at (-1,-0.5) {$a_2$};
\node[scale=0.8] at (1,-0.5) {$a_3$};
\node[scale=0.8] at (4.8,-0.5) {$a_4$};
\node[scale=0.8] at (-0.91,-1.5) {$\delta_1$};
\node[scale=0.8] at (0.96,-1.5) {$\delta_3$};
\node[scale=0.8, rotate=45] at (-1.5,0.7) {$\delta_2$};
\node[scale=0.8, rotate=-15] at (2.3,1.15) {$\delta_4$};
\end{scope}
\end{tikzpicture}
\caption{$4$--holed torus curves  and the embedding $\Sigma_1^4 \hookrightarrow \Sigma_3^1$} \label{fig:(3,5)proof}
\end{figure}
\smallskip

Thus,  we have the following relator in $\M(\Sigma_3^1)$ with $\sigma=-4$:
\[
 1 =  t_{\delta_1}^{-1}  t_{\delta_2}^{-1}  t_{ \delta_3 }^{-1}  t_{\delta_4}^{-1}
	 (t_{a_0}t_{a_1}t_{a_3}t_{a_0}t_{a_2}t_{a_4})^2 =  t_{a_0}  t_{z_1} t_{z_2}t_{z_3}  (t_{a_1}t_{a_2}t_{a_3}t_{a_4})^2 \,  ,
\]
 where $z_1=t_{a_1}t_{a_3}(a_0)$, $z_2=t_{a_1}t_{a_2}t_{a_3}t_{a_4}(a_0)$ and 
 $z_3=  t_{a_1}t_{a_1}t_{a_2}t_{a_3}t_{a_3}t_{a_4}(a_0)$.  Using \eqref{1comm} and  Lemma~\ref{lem:g=3de3comm},
 we can change this into   
 \begin{eqnarray*}
 1 &=& (t_{\delta_1}^{-1} t_{a_0} \, t_{z_1}t_{\delta_2}^{-1}) (t_{\delta_3}^{-1}t_{z_2} \, t_{z_3}t_{\delta_4}^{-1})   (t_{a_1}t_{a_2}t_{a_3}t_{a_4})^2\\
  &=& C_4 C_5(t_{a_1}t_{a_2}t_{a_3}t_{a_4})^2 \cdot (t_{a_1}t_{a_2}t_{a_3}t_{a_4})^{-2}C_1C_2C_3  \\
 &=& C_4C_5\cdot C_1 C_2 C_3 \, ,
 \end{eqnarray*}
or equivalently,  into the relator 
\[  
C_1 C_2 C_3 C_4 C_5 =1 \ \ \text{ in } \ \M(\Sigma_3^1) \, ,
\]
with signature  $\sigma=-4+8=4$.  This prescribes a $\Sigma_3$--bundle over $\Sigma_5$  with $\sigma=4$ and a section of self-intersection zero.  In turn, we get $\Sigma_g$--bundle  over $\Sigma_h$ with $\sigma=4$ (and a section of self-intersection zero) for any $g \geq 3$ and $h \geq 5$.

\bigskip
 \noindent \underline{\textit{$h=4$ and $g \geq 5$\,}}: Take an embedding $\Sigma_1^4 \hookrightarrow \Sigma_5^1$ so that the $c_i$ curves of the $4$--holed torus relation (c.f. Figure~\ref{fig:(3,5)proof}) are as shown on the left hand-side of Figure~\ref{fig:(5,4)icin}, whereas the boundary curves $\delta_j$ are mapped to the separating curves $d_j$.

\medskip
\begin{figure}[h]
\begin{tikzpicture}[scale=0.88]
\begin{scope} [xshift=0cm, yshift=0cm, scale=0.74]
   \draw[very thick, xshift=0cm] (0,0) circle [radius=1cm];
   \draw[very thick ] (2.94, 0.98) arc (18.4:71.5:3.1);
   \draw[very thick, rotate=90 ] (2.94, 0.98) arc (18.4:71.5:3.1);
   \draw[very thick, rotate=180 ] (2.94, 0.98) arc (18.4:71.5:3.1);
   \draw[very thick, rotate=-90 ] (2.94, 0.98) arc (18.4:71.5:3.1);
    \draw[thick, red, xshift=0cm] (0,0) circle [radius=1.3cm];
    \draw[thick, red, rounded corners=3pt, rotate=0] (-1.02,2.93)..controls (-0.5,2.8) and (0.5,2.8) ..(1.02,2.93);
   \draw[thick, red, dashed, rounded corners=3pt, rotate=0] (-1,2.97).. controls (-0.5,3) and (0.5,3) ..(1.,2.97);
  \draw[very thick] (-1, 2.93)--(-1 ,3.42);
  \draw[very thick] (1, 2.93)--(1 ,3.42); 
  \draw[very thick] (0.695,3.4) ellipse (0.305 and 0.12);
  \draw[very thick] (-0.695,3.4) ellipse (0.305 and 0.12);
  \draw[very thick ] (0.39, 3.4) arc (-14:-170:0.4);
  \draw[very thick, gray] (0,4.6) ellipse (0.28 and 0.07);
  \draw[very thick, gray] (1,3.8)..controls (0.98, 4.3) and (0.5,4.55)..(0.28,4.59);
  \draw[very thick, gray] (-1,3.8)..controls (-0.98, 4.3) and (-0.5,4.55)..(-0.28,4.59);
 \draw[very thick, gray] (-0.385 ,3.8)..controls (-0.3,4.2) and (0.3,4.2)..(0.385,3.8);
\draw[very thick, gray] (0.384 ,3.815)..controls (0.394,3.63) and (0.99,3.63)..(1,3.82);
\draw[very thick, dashed, gray] (0.384 ,3.815)..controls (0.394,4) and (0.99,4)..(1,3.82);
\draw[very thick, gray] (-0.384 ,3.815)..controls (-0.394,3.63) and (-0.99,3.63)..(-1,3.82);
\draw[very thick, dashed, gray] (-0.384 ,3.815)..controls (-0.394,4) and (-0.99,4)..(-1,3.82);
\draw[very thick, gray, rotate=90 ] (-1 ,3.8)..controls (-0.9,4.9) and (0.9,4.9)..(1,3.8);
 \draw[very thick, gray, rotate=90] (-0.385 ,3.8)..controls (-0.3,4.2) and (0.3,4.2)..(0.385,3.8);
\draw[very thick, gray, rotate=90] (0.384 ,3.815)..controls (0.394,3.63) and (0.99,3.63)..(1,3.82);
\draw[very thick, dashed, gray, rotate=90] (0.384 ,3.815)..controls (0.394,4) and (0.99,4)..(1,3.82);
\draw[very thick, gray, rotate=90] (-0.384 ,3.815)..controls (-0.394,3.63) and (-0.99,3.63)..(-1,3.82);
\draw[very thick, dashed, gray, rotate=90] (-0.384 ,3.815)..controls (-0.394,4) and (-0.99,4)..(-1,3.82);
\draw[very thick, gray, rotate=180 ] (-1 ,3.8)..controls (-0.9,4.9) and (0.9,4.9)..(1,3.8);
 \draw[very thick, gray, rotate=180] (-0.385 ,3.8)..controls (-0.3,4.2) and (0.3,4.2)..(0.385,3.8);
\draw[very thick, gray, rotate=180] (0.384 ,3.815)..controls (0.394,3.63) and (0.99,3.63)..(1,3.82);
\draw[very thick, dashed, gray, rotate=180] (0.384 ,3.815)..controls (0.394,4) and (0.99,4)..(1,3.82);
\draw[very thick, gray, rotate=180] (-0.384 ,3.815)..controls (-0.394,3.63) and (-0.99,3.63)..(-1,3.82);
\draw[very thick, dashed, gray, rotate=180] (-0.384 ,3.815)..controls (-0.394,4) and (-0.99,4)..(-1,3.82);
\draw[very thick, gray, rotate=-90 ] (-1 ,3.8)..controls (-0.9,4.9) and (0.9,4.9)..(1,3.8);
 \draw[very thick, gray, rotate=-90] (-0.385 ,3.8)..controls (-0.3,4.2) and (0.3,4.2)..(0.385,3.8);
\draw[very thick, gray, rotate=-90] (0.384 ,3.815)..controls (0.394,3.63) and (0.99,3.63)..(1,3.82);
\draw[very thick, dashed, gray, rotate=-90] (0.384 ,3.815)..controls (0.394,4) and (0.99,4)..(1,3.82);
\draw[very thick, gray, rotate=-90] (-0.384 ,3.815)..controls (-0.394,3.63) and (-0.99,3.63)..(-1,3.82);
\draw[very thick, dashed, gray, rotate=-90] (-0.384 ,3.815)..controls (-0.394,4) and (-0.99,4)..(-1,3.82);

   \draw[thick, red, rounded corners=3pt, rotate=90] (-1.02,2.93)..controls (-0.5,2.8) and (0.5,2.8) ..(1.02,2.93);
   \draw[thick, red, dashed, rounded corners=3pt, rotate=90] (-1,2.97).. controls (-0.5,3) and (0.5,3) ..(1.,2.97);
   \draw[very thick, rotate=90 ] (-1, 2.93)--(-1 ,3.42);
   \draw[very thick, rotate=90 ] (1, 2.93)--(1 ,3.42); 
   \draw[very thick, rotate=90] (0.695,3.4) ellipse (0.305 and 0.12);
  \draw[very thick, rotate=90] (-0.695,3.4) ellipse (0.305 and 0.12);
  \draw[very thick, rotate=90  ] (0.39, 3.4) arc (-14:-170:0.4);

    \draw[thick, red, rounded corners=3pt, rotate=180] (-1.02,2.93)..controls (-0.5,2.8) and (0.5,2.8) ..(1.02,2.93);
   \draw[thick, red, dashed, rounded corners=3pt, rotate=180] (-1,2.97).. controls (-0.5,3) and (0.5,3) ..(1.,2.97);
   \draw[very thick, rotate=180] (-1, 2.93)--(-1 ,3.42);
   \draw[very thick,  rotate=180] (1, 2.93)--(1 ,3.42); 
   \draw[very thick, rotate=180] (0.695,3.4) ellipse (0.305 and 0.12);
  \draw[very thick, rotate=180] (-0.695,3.4) ellipse (0.305 and 0.12);
  \draw[very thick, rotate=180 ] (0.39, 3.4) arc (-14:-170:0.4);
    \draw[thick, red, rounded corners=3pt, rotate=270] (-1.02,2.93)..controls (-0.5,2.8) and (0.5,2.8) ..(1.02,2.93);
   \draw[thick, red, dashed, rounded corners=3pt, rotate=270] (-1,2.97).. controls (-0.5,3) and (0.5,3) ..(1.,2.97);
   \draw[very thick, rotate=-90 ] (-1, 2.93)--(-1 ,3.42);
   \draw[very thick, rotate=-90 ] (1, 2.93)--(1 ,3.42); 
   \draw[very thick, rotate=-90] (0.695,3.4) ellipse (0.305 and 0.12);
  \draw[very thick, rotate=-90] (-0.695,3.4) ellipse (0.305 and 0.12);
  \draw[very thick, rotate=-90] (0.39, 3.4) arc (-14:-170:0.4);
 \draw[thick, red, rounded corners=3pt, rotate=45] (1,0).. controls (1.6, -0.1) and (2.5,-0.1) ..(3.12,0);
 \draw[thick, red,  dashed, rounded corners=3pt, rotate=45] (1,0.02).. controls (1.6, 0.1) and (2.5,0.1) ..(3.12,0.02);
 \draw[thick, red, rounded corners=3pt, rotate=135] (1,0).. controls (1.6, -0.1) and (2.5,-0.1) ..(3.12,0);
 \draw[thick, red,  dashed, rounded corners=3pt, rotate=135] (1,0.02).. controls (1.6, 0.1) and (2.5,0.1) ..(3.12,0.02);
 \draw[thick, red, rounded corners=3pt, rotate=225] (1,0).. controls (1.6, -0.1) and (2.5,-0.1) ..(3.12,0);
 \draw[thick, red,  dashed, rounded corners=3pt, rotate=225] (1,0.02).. controls (1.6, 0.1) and (2.5,0.1) ..(3.12,0.02);
 \draw[thick, red, rounded corners=3pt, rotate=315] (1,0).. controls (1.6, -0.1) and (2.5,-0.1) ..(3.12,0);
 \draw[thick, red,  dashed, rounded corners=3pt, rotate=315] (1,0.02).. controls (1.6, 0.1) and (2.5,0.1) ..(3.12,0.02);
 \node[scale=0.8] at (0,1.6) {$c_0$};
 \node[scale=0.8] at (-2.4,2.5) {$c_1$};
 \node[scale=0.8] at (-2.4,-2.5) {$c_2$};
 \node[scale=0.8] at (2.4,-2.5) {$c_3$};
 \node[scale=0.8] at (2.4,2.5) {$c_4$};
 \node[scale=0.8] at (-3.6,1.4) {$\delta_1$};
 \node[scale=0.8] at (-3.6,-1.4) {$\delta_2$};
 \node[scale=0.8] at (-1.4, -3.6) {$\delta_3$};
 \node[scale=0.8] at (1.4, -3.6) {$\delta_4$};
 \node[scale=0.8] at (3.6,-1.4) {$\delta_5$};
 \node[scale=0.8] at (3.6,1.4) {$\delta_6$};
 \node[scale=0.8] at (1.4, 3.6) {$\delta_7$};
 \node[scale=0.8] at (-1.4, 3.6) {$\delta_8$};
 \node[scale=0.8] at (-2.5,0) {$d_1$};
 \node[scale=0.8] at (0,-2.5) {$d_2$};
 \node[scale=0.8] at (2.5,0) {$d_3$};
 \node[scale=0.8] at (0,2.5) {$d_4$};

\end{scope}

\begin{scope} [xshift=8.1cm, yshift=0cm, scale=0.74]
   \draw[very thick, xshift=0cm] (0,0) circle [radius=1cm];
   \draw[very thick ] (2.94, 0.98) arc (18.4:71.5:3.1);
   \draw[very thick, rotate=90 ] (2.94, 0.98) arc (18.4:71.5:3.1);
   \draw[very thick, rotate=180 ] (2.94, 0.98) arc (18.4:71.5:3.1);
   \draw[very thick, rotate=-90 ] (2.94, 0.98) arc (18.4:71.5:3.1);
     \draw[very thick] (-1, 2.93)--(-1 ,3.42);
   \draw[very thick] (1, 2.93)--(1 ,3.42); 
  \draw[very thick] (0.695,3.4) ellipse (0.305 and 0.12);
  \draw[very thick] (-0.695,3.4) ellipse (0.305 and 0.12);
  \draw[very thick ] (0.39, 3.4) arc (-14:-170:0.4);
   \draw[very thick, rotate=90 ] (-1, 2.93)--(-1 ,3.42);
   \draw[very thick, rotate=90 ] (1, 2.93)--(1 ,3.42); 
   \draw[very thick, rotate=90] (0.695,3.4) ellipse (0.305 and 0.12);
  \draw[very thick, rotate=90] (-0.695,3.4) ellipse (0.305 and 0.12);
  \draw[very thick, rotate=90  ] (0.39, 3.4) arc (-14:-170:0.4);
   \draw[very thick, rotate=180] (-1, 2.93)--(-1 ,3.42);
   \draw[very thick,  rotate=180] (1, 2.93)--(1 ,3.42); 
   \draw[very thick, rotate=180] (0.695,3.4) ellipse (0.305 and 0.12);
  \draw[very thick, rotate=180] (-0.695,3.4) ellipse (0.305 and 0.12);
  \draw[very thick, rotate=180 ] (0.39, 3.4) arc (-14:-170:0.4);
   \draw[very thick, rotate=-90 ] (-1, 2.93)--(-1 ,3.42);
   \draw[very thick, rotate=-90 ] (1, 2.93)--(1 ,3.42); 
   \draw[very thick, rotate=-90] (0.695,3.4) ellipse (0.305 and 0.12);
  \draw[very thick, rotate=-90] (-0.695,3.4) ellipse (0.305 and 0.12);
  \draw[very thick, rotate=-90] (0.39, 3.4) arc (-14:-170:0.4);

  \draw[very thick, gray] (0,4.6) ellipse (0.28 and 0.07);
  \draw[very thick, gray] (1,3.8)..controls (0.98, 4.3) and (0.5,4.55)..(0.28,4.59);
  \draw[very thick, gray] (-1,3.8)..controls (-0.98, 4.3) and (-0.5,4.55)..(-0.28,4.59);
 \draw[very thick, gray] (-0.385 ,3.8)..controls (-0.3,4.2) and (0.3,4.2)..(0.385,3.8);
\draw[very thick, gray] (0.384 ,3.815)..controls (0.394,3.63) and (0.99,3.63)..(1,3.82);
\draw[very thick, dashed, gray] (0.384 ,3.815)..controls (0.394,4) and (0.99,4)..(1,3.82);
\draw[very thick, gray] (-0.384 ,3.815)..controls (-0.394,3.63) and (-0.99,3.63)..(-1,3.82);
\draw[very thick, dashed, gray] (-0.384 ,3.815)..controls (-0.394,4) and (-0.99,4)..(-1,3.82);
\draw[very thick, gray, rotate=90 ] (-1 ,3.8)..controls (-0.9,4.9) and (0.9,4.9)..(1,3.8);
 \draw[very thick, gray, rotate=90] (-0.385 ,3.8)..controls (-0.3,4.2) and (0.3,4.2)..(0.385,3.8);
\draw[very thick, gray, rotate=90] (0.384 ,3.815)..controls (0.394,3.63) and (0.99,3.63)..(1,3.82);
\draw[very thick, dashed, gray, rotate=90] (0.384 ,3.815)..controls (0.394,4) and (0.99,4)..(1,3.82);
\draw[very thick, gray, rotate=90] (-0.384 ,3.815)..controls (-0.394,3.63) and (-0.99,3.63)..(-1,3.82);
\draw[very thick, dashed, gray, rotate=90] (-0.384 ,3.815)..controls (-0.394,4) and (-0.99,4)..(-1,3.82);
\draw[very thick, gray, rotate=180 ] (-1 ,3.8)..controls (-0.9,4.9) and (0.9,4.9)..(1,3.8);
 \draw[very thick, gray, rotate=180] (-0.385 ,3.8)..controls (-0.3,4.2) and (0.3,4.2)..(0.385,3.8);
\draw[very thick, gray, rotate=180] (0.384 ,3.815)..controls (0.394,3.63) and (0.99,3.63)..(1,3.82);
\draw[very thick, dashed, gray, rotate=180] (0.384 ,3.815)..controls (0.394,4) and (0.99,4)..(1,3.82);
\draw[very thick, gray, rotate=180] (-0.384 ,3.815)..controls (-0.394,3.63) and (-0.99,3.63)..(-1,3.82);
\draw[very thick, dashed, gray, rotate=180] (-0.384 ,3.815)..controls (-0.394,4) and (-0.99,4)..(-1,3.82);
\draw[very thick, gray, rotate=-90 ] (-1 ,3.8)..controls (-0.9,4.9) and (0.9,4.9)..(1,3.8);
 \draw[very thick, gray, rotate=-90] (-0.385 ,3.8)..controls (-0.3,4.2) and (0.3,4.2)..(0.385,3.8);
\draw[very thick, gray, rotate=-90] (0.384 ,3.815)..controls (0.394,3.63) and (0.99,3.63)..(1,3.82);
\draw[very thick, dashed, gray, rotate=-90] (0.384 ,3.815)..controls (0.394,4) and (0.99,4)..(1,3.82);
\draw[very thick, gray, rotate=-90] (-0.384 ,3.815)..controls (-0.394,3.63) and (-0.99,3.63)..(-1,3.82);
\draw[very thick, dashed, gray, rotate=-90] (-0.384 ,3.815)..controls (-0.394,4) and (-0.99,4)..(-1,3.82);

  \draw[thick, red, dashed, rounded corners=3pt, rotate=55] (1,0).. controls (1.6, -0.2) and (2.5,-0.2) ..(3.12,0);
 \draw[thick, red, dashed, rounded corners=3pt, rotate=0] (-1,2.95).. controls (-0.8, 2.95) and (-0.5,3) ..(-0.23,3.2);
 \draw[thick, red,  rounded corners=3pt, rotate=0] (-0.23,3.16).. controls (0.5, 2.5) and (1,2.5) ..(1.75,2.55);
 \draw[thick, red,  rounded corners=3pt, rotate=0] (-1,2.92).. controls (-0.5, 2.4) and (0,2.4) ..(0.52,0.84);
 
 \draw[thick, red, dashed, rounded corners=3pt, rotate=145] (1,0).. controls (1.6, -0.2) and (2.5,-0.2) ..(3.12,0);
 \draw[thick, red, dashed, rounded corners=3pt, rotate=90] (-1,2.95).. controls (-0.8, 2.95) and (-0.5,3) ..(-0.23,3.2);
 \draw[thick, red,  rounded corners=3pt, rotate=90] (-0.23,3.16).. controls (0.5, 2.5) and (1,2.5) ..(1.75,2.55);
 \draw[thick, red,  rounded corners=3pt, rotate=90] (-1,2.92).. controls (-0.5, 2.4) and (0,2.4) ..(0.52,0.84);

 \draw[thick, red, dashed, rounded corners=3pt, rotate=235] (1,0).. controls (1.6, -0.2) and (2.5,-0.2) ..(3.12,0);
 \draw[thick, red, dashed, rounded corners=3pt, rotate=180] (-1,2.95).. controls (-0.8, 2.95) and (-0.5,3) ..(-0.23,3.2);
 \draw[thick, red,  rounded corners=3pt, rotate=180] (-0.23,3.16).. controls (0.5, 2.5) and (1,2.5) ..(1.75,2.55);
 \draw[thick, red,  rounded corners=3pt, rotate=180] (-1,2.92).. controls (-0.5, 2.4) and (0,2.4) ..(0.52,0.84);

 \draw[thick, red, dashed, rounded corners=3pt, rotate=325] (1,0).. controls (1.6, -0.2) and (2.5,-0.2) ..(3.12,0);
 \draw[thick, red, dashed, rounded corners=3pt, rotate=270] (-1,2.95).. controls (-0.8, 2.95) and (-0.5,3) ..(-0.23,3.2);
 \draw[thick, red,  rounded corners=3pt, rotate=270] (-0.23,3.16).. controls (0.5, 2.5) and (1,2.5) ..(1.75,2.55);
 \draw[thick, red,  rounded corners=3pt, rotate=270] (-1,2.92).. controls (-0.5, 2.4) and (0,2.4) ..(0.52,0.84);

\draw[thick, blue, rounded corners=3pt, rotate=0] (1,0).. controls (1.6, -0.1) and (2.5,-0.1) ..(3.12,0);
 \draw[thick, blue,  dashed, rounded corners=3pt, rotate=0] (1,0.02).. controls (1.6, 0.1) and (2.5,0.1) ..(3.12,0.02);
 \draw[thick, blue, rounded corners=3pt, rotate=90] (1,0).. controls (1.6, -0.1) and (2.5,-0.1) ..(3.12,0);
 \draw[thick, blue,  dashed, rounded corners=3pt, rotate=90] (1,0.02).. controls (1.6, 0.1) and (2.5,0.1) ..(3.12,0.02);
 \draw[thick, blue, rounded corners=3pt, rotate=180] (1,0).. controls (1.6, -0.1) and (2.5,-0.1) ..(3.12,0);
 \draw[thick, blue,  dashed, rounded corners=3pt, rotate=180] (1,0.02).. controls (1.6, 0.1) and (2.5,0.1) ..(3.12,0.02);
 \draw[thick, blue, rounded corners=3pt, rotate=270] (1,0).. controls (1.6, -0.1) and (2.5,-0.1) ..(3.12,0);
 \draw[thick, blue,  dashed, rounded corners=3pt, rotate=270] (1,0.02).. controls (1.6, 0.1) and (2.5,0.1) ..(3.12,0.02);
 
 \node[scale=0.8] at (-2.9,1.9) {$x_1$};
 \node[scale=0.8] at (-2,-2.8) {$x_2$};
 \node[scale=0.8] at (2.9,-1.9) {$x_3$};
 \node[scale=0.8] at (2,2.8) {$x_4$};
 \node[scale=0.8] at (-1.4,-0.4) {$y_1$};
 \node[scale=0.8] at (0.45,-1.5) {$y_2$};
 \node[scale=0.8] at (1.4,0.4) {$y_3$};
 \node[scale=0.8] at (-0.45,1.5) {$y_4$};
\end{scope}
\end{tikzpicture}
\caption{The curves on $\Sigma_5^1$.} \label{fig:(5,4)icin}
\end{figure}
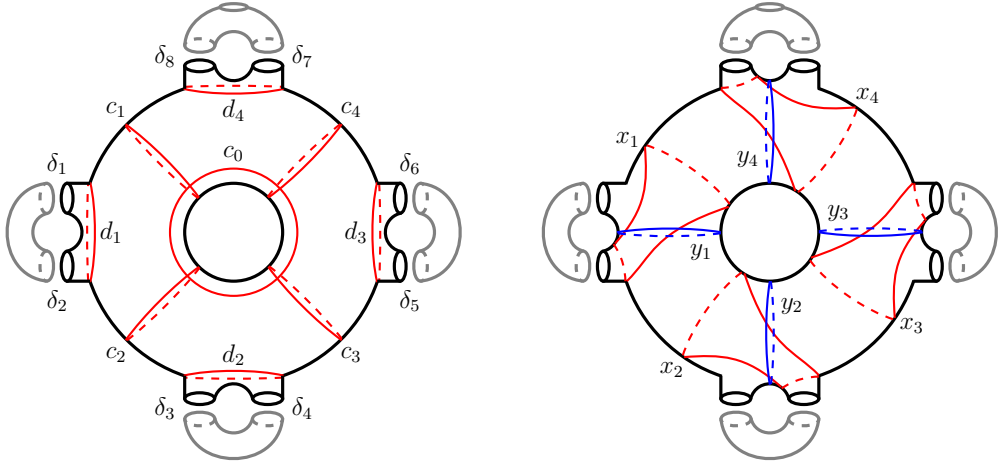
\smallskip

In the following, let us denote $t_{c_i}$ by $t_i$ for simplicity.   From  the $4$-holed torus relation, we have the following in $\M(\Sigma_5^1)$:
 \begin{eqnarray*}
  t_{d_1}  t_{d_2}  t_{d_3 }  t_{d_4} 
	&=&  t_0t_1t_3 t_0t_2t_4t_0t_1t_3t_0t_2t_4\\
	&=&  t_0    t_0^{t_1t_3 t_0t_2t_4}  t_1t_3 t_0t_2t_4 t_1t_3t_0t_2t_4\\
	&=&  t_0    t_0^{t_1t_3 t_0t_2t_4}   t_0^{t_1t_3} t_1^2t_3^2 t_2t_4  t_0t_2t_4\\
	&=&  t_0    t_0^{t_1t_3 t_0t_2t_4}   t_0^{t_1t_3}   t_0^{t_1^2t_3^2 t_2t_4}t_1^2t_2^2 t_3^2t_4^2\\
	&=& t_0 t_{v_0} t_{v_1}t_{v_2} t_1^2t_2^2 t_3^2t_4^2,
 \end{eqnarray*}
where $v_0=t_1t_3t_0t_2t_4 (c_0)$, $v_1= t_1t_3(c_0)$ and 
 $v_2=t_1^2t_3^2t_2 t_4(c_0)$.  Note that $c_0$ and $v_0$ cobound a  subsurface $\Sigma_2^2$ in $\Sigma_5^1$.  We multiply both sides of this equality by 
 $t_{\delta_1}t_{\delta_2}\cdots t_{\delta_8}$ and, for $i=1,2,3,4$,  apply the four lantern relations of the form
 \[
 t_it_{i+1}t_{\delta_{2i-1}}t_{\delta_{2i}} =t_{y_i}t_{x_i}t_{d_i} \,  , 
 \]
 with the agreement that $t_5=t_1$ to get 
  \begin{align*}
   t_{d_1}  t_{d_2}  t_{d_3 }  t_{d_4}& t_{\delta_1}t_{\delta_2}t_{\delta_3}t_{\delta_4}t_{\delta_5}t_{\delta_6}t_{\delta_7} t_{\delta_8}\\
  &= t_0 t_{v_0} t_{v_1}t_{v_2} \cdot t_1t_2t_{\delta_1}t_{\delta_2}
      	\cdot t_2t_3 t_{\delta_3}t_{\delta_4} \cdot t_3t_4t_{\delta_5}t_{\delta_6}
	\cdot t_4t_1 t_{\delta_7}t_{\delta_8}\\
  &= t_0 t_{v_0} t_{v_1}t_{v_2} \cdot t_{y_1}t_{x_1}t_{d_1}
      	\cdot t_{y_2}t_{x_2}t_{d_2} \cdot t_{y_3}t_{x_3}t_{d_3}
	\cdot t_{y_4}t_{x_4}t_{d_4}\\
  &= t_0 t_{v_0} t_{v_1}t_{v_2} \cdot t_{y_1}t_{x_1}
      	\cdot t_{y_2}t_{x_2} \cdot t_{y_3}t_{x_3}
	\cdot t_{y_4}t_{x_4}\cdot t_{d_1} t_{d_2}  t_{d_3} t_{d_4}.
 \end{align*}
Now,  cancelling all $t_{d_i}$ gives  
  \begin{align*}
	t_{\delta_1}t_{\delta_2}t_{\delta_3}t_{\delta_4}t_{\delta_5}t_{\delta_6}
		t_{\delta_7} t_{\delta_8}
  	&= t_0 t_{v_0} t_{v_1}t_{v_2} \cdot t_{y_1}t_{x_1}
      		\cdot t_{y_2}t_{x_2} \cdot t_{y_3}t_{x_3}\cdot t_{y_4}t_{x_4}\\
  	&= t_0 t_{v_0} t_{v_1}t_{v_2} \cdot t_{y_1} t_{y_2}t_{y_3} t_{y_4}
      		\cdot t_{x_1} t_{x_2}t_{x_3}t_{x_4} \, .
 \end{align*}
 Supported in $\Sigma_1^8$, this is \emph{an $8$-holed torus relation} (cf. \cite{KorkmazOzbagci2,  Hamada}).  

 Hence,  using \eqref{1comm} and Lemma~\ref{lem:g=5}, we have  
 \begin{align*}
  	1&= t_0 t_{v_0} t_{v_1}t_{v_2} 
	\cdot t_{y_1} t_{y_2}t_{y_3} t_{y_4} t_{\delta_1}^{-1}t_{\delta_3}^{-1}t_{\delta_5}^{-1}t_{\delta_7}^{-1}
      		\cdot t_{x_1} t_{x_2}t_{x_3}t_{x_4} t_{\delta_2}^{-1}t_{\delta_4}^{-1}t_{\delta_6}^{-1}t_{\delta_8}^{-1}\\
  	&= (t_0 t_{v_0}t_{\delta_1}^2) \cdot (t_{v_1}t_{\delta_1}^{-1} t_{\delta_1}^{-1} t_{v_2})\\
	&\hspace*{2.5cm} \cdot (t_{y_1} t_{y_2}t_{y_3} t_{y_4} t_{\delta_1}^{-1}t_{\delta_3}^{-1}t_{\delta_3}^{-1}t_{\delta_7}^{-1}
      		\cdot t_{x_1} t_{x_2}t_{x_3}t_{x_4} t_{\delta_2}^{-1}t_{\delta_4}^{-1}t_{\delta_6}^{-1}t_{\delta_8}^{-1})\\
	&= C_1C_2\cdot C_3 \cdot C_4 \, ,
 \end{align*}
in $\M(\Sigma_5^1)$,  where all $C_i$ are commutators.  Here we have used the fact that there
are two diffeomorphisms of $\Sigma_5^1$ such that 
one maps $(v_1,\delta_1)$ to $(\delta_1, v_2)$, and the other maps
$(y_1,y_2,y_3,y_4,\delta_1,\delta_3,\delta_5,\delta_7)$ to 
$(\delta_2,\delta_4,\delta_6,\delta_8,x_1,x_2,x_3,x_4)$. (We note that $\delta_i=\delta_{i+1}$
for $i=1,3,5$.)

Keeping track of the involved basic relators,  one can easily see that the $8$-holed torus relator has $\sigma=-4+4(+1)=0$. Since the relator  $t_0^{-1} t_{v_0}^{-1}t_{\delta_1}^{-2}\cdot C_1C_2=1$ from Lemma~\ref{lem:g=5} has  $\sigma=4$,  we then conclude that the relator 
$C_1C_2C_3C_4=1$ in $\M(\Sigma_5^1)$  has $\sigma=0+4=4$.  This prescribes a $\Sigma_5$-- bundle over $\Sigma_4$  with $\sigma=4$ and a section of self-intersection zero.  In turn, we get $\Sigma_g$--bundle  over $\Sigma_4$ with $\sigma=4$ (and a section of self-intersection zero) for any $g \geq 5$.

\bigskip
 \noindent \underline{\textit{$g=4$ and $h=4$}}:  We next construct a $\Sigma_4$--bundle over $\Sigma_4$ \emph{without a section of self-intersection zero}.  By \cite[Theorem 1.3]{CataneseEtal}, there exists a genus--$4$  semi-simple holomorphic fibration over $T^2$,  with exactly two singular fibers, each consisting of a genus--$2$ curve and an elliptic curve meeting transversally in two points.  The total space of this fibration is smoothly a product $\Sigma_2 \times \Sigma_2$,  so it has $\sigma=0$.  Thus, we have a relator
\[
t_{a_1} t_{a_2} t_{b_1} t_{b_2}  C_4 =1 \ \ \ \text{ in } \M(\Sigma_4)
\]
with $\sigma=0$, where $C_4$ is a commutator.  Since there exists an $F \in \textrm{Diff}^+(\Sigma_4)$ with  $F(a_i)=b_i$,   we can rewrite the above relator as
\begin{align*}
1
  	&= t_{a_1} t_{a_2} (t_{a_1} t_{a_2})^F  C_4 \\
  	&= (t_{a_1} t_{a_2})^2 [(t_{a_1} t_{a_2})^{-1}, F] \,  C_4 \\
  	 &= (t_{a_1}^2 t_{a_2}^2) \cdot C_3 C_4 \, ,
 \end{align*}
where we have set $C_3:= [(t_{a_1} t_{a_2})^{-1}, F] $.  We can now invoke Lemma~\ref{lem:genus=2} to replace the first factor and get the relator
\[
C_1 C_2 C_3  C_4 =1 \ \ \ \text{ in } \M(\Sigma_4)
\]
which has $\sigma=0+4=4$.

\bigskip
 \noindent \underline{\textit{$h=3$, $g \geq 9$ and $h=2$,  $g \geq 15$}}: \,Consider the $\Sigma_3$--bundle over $\Sigma_5$ we constructed above,  which has $\sigma=4$ and a section of self-intersection zero.  By Theorem~\ref{thm:tsuboish},  we can derive two more bundles from it: a $\Sigma_{15}$--bundle over $\Sigma_2$ and a $\Sigma_{9}$--bundle over $\Sigma_3$,  both also with $\sigma=4$ and sections of self-intersection zero.  In turn, we get a $\Sigma_g$--bundle over $\Sigma_2$ with $\sigma=4$ for any $g \geq 15$, and a $\Sigma_g$--bundle over $\Sigma_3$ with $\sigma=4$ for any $g \geq 9$.

\smallskip
All the surface bundles we constructed can be equipped with a Thurston symplectic form.   This completes the proof of Theorem~\ref{main}.  $\QED$

\smallskip
\begin{remark} \label{rk:prior}
In \cite{EKKOS},  Endo, Kotschick, Ozbagci, Stipsicz and the second author constructed surface bundles with positive signatures for all $g \geq 3$ and $h \geq 9$.  On the other hand,  Bryan and Donagi established in  \cite{BryanDonagi} that the base genus of a positive signature surface bundle could be as small as $2$.  These results were partially improved by Lee in \cite{Lee}, and most successfully by Monden in \cite{Monden}, who in particular produced positive signature surface bundles for all $g \geq 39$ and $h = 2$.
\end{remark}

\begin{remark} \label{Nonholomorphic}
Following the recipe of \cite{BaykurNonholomorphic},  our surface bundles yield further examples of \emph{non-holomorphic} surface bundles over surfaces with positive signatures at least for every  \,$g \geq 16$, $h=2$; \ $g \geq 10$, $h=3$; \ $g \geq 6$, $h =4$; $g \geq 3 $, $h = 5$ and $g \geq 3 $, $h  \geq 6$.  In particular,  we answer the question of existence of non-holomorphic \mbox{$\Sigma_g$--bundles} over $\Sigma_h$ with $\sigma \neq 0$ for all but finitely many pairs of $(g,h)$.  The signatures for all of these non-holomorphic examples can be chosen to be $4$.  In fact, to the best of our knowledge,  there are no examples of holomorphic surface bundles with   $\sigma=4$.  Is there an obstruction?
\end{remark}

\begin{remark} \label{highersignature}
While the techniques of our paper can possibly be employed to generate surface bundles with high \emph{Chern slope} $c_1^2 / c_2$ (equivalently, high $\sigma/ \eu$ ratio) we do not currently have any examples with slopes higher than the ones obtained by Catanese and Rollenske in  \cite{CataneseRollenske},  making it all the more curious  whether the Chern numbers of any (symplectic) surface bundle over a surface always satisfy $c_1^2 / c_2 \leq 2+ 2/3$. 
\end{remark}

\vspace{0.2in}
\noindent \textit{Acknowledgements. }
The authors would like to thank the Max Planck Institute for Mathematics in Bonn,  where this work was completed during their stay. The first author was  supported by the Simons Foundation Grant 634309 and the National Science Foundation grant DMS-2005327.

\end{document}